\documentclass[12pt, twosideside]{article}

\usepackage{amsmath,amsthm,amsfonts,amssymb, latexsym, eepic, xypic, fancyhdr,cancel, geometry}
\usepackage[dvips]{epsfig}
\usepackage[usenames]{color}

\def\titlerunning#1{\gdef\titrun{#1}}
\makeatletter
\def\author#1{\gdef\autrun{\def\and{\unskip, }#1}\gdef\@author{#1}}
\def\address#1{{\def\and{\\\hspace*{18pt}}\renewcommand{\thefootnote}{}%
\footnote {#1}}%
\markboth{\autrun}{\titrun}}
\makeatother
\def\email#1{e-mail: #1}

\newtheorem{thm}{Theorem}

\newtheorem*{claim}{Claim}
\newtheorem{lem}[thm]{Lemma}
\newtheorem{prop}[thm]{Proposition}
\newtheorem{cor}[thm]{Corollary}

\theoremstyle{definition}
\newtheorem{defin}{Definition}
\newtheorem*{notation}{Notation}

\newtheorem*{remark}{Remark}

\DeclareMathOperator{\Div}{div}
\DeclareMathOperator{\spt}{spt}

\DeclareMathOperator{\tr}{tr}
\DeclareMathOperator{\graph}{graph}

\def\rr{\mathbb{R}}

\def\isom{\cong}

\def\area{\mathcal{H}^{n-1}}
\def\barea{\mathcal{H}^{n-2}}
\def\F{\mathcal{F}}
\def\C{C}

\begin{document}

\baselineskip=17pt

\title{The spacetime positive mass theorem in dimensions less than eight}

\titlerunning{The spacetime positive mass theorem in dimensions less than eight}

\author{
Michael Eichmair
\and
Lan-Hsuan Huang
 \and
Dan A.\ Lee
\and
Richard Schoen
}

\date{\today}

\maketitle

\address{M. Eichmair: University of Vienna; \email{michael.eichmair@univie.ac.at} \and L.-H. Huang: University of Connecticut; \email {lan-hsuan.huang@uconn.edu} \and D. A. Lee: Queens College and CUNY Graduate Center; \email{dan.lee@qc.cuny.edu} \and R. Schoen: University of California, Irvine; \email {rschoen@math.uci.edu}}

\begin{abstract}
We prove the spacetime positive mass theorem in dimensions less than eight.  This theorem asserts that for any asymptotically flat initial data set that satisfies the dominant energy condition, the inequality $E\ge |P|$ holds, where $(E,P)$ is the ADM energy-momentum vector. Previously, this theorem was  only known for spin manifolds \cite{Witten:1981}. Our approach is a modification of the minimal hypersurface technique that was used by the last named author and S.-T.\ Yau to establish the time-symmetric case of this theorem \cite{Schoen-Yau:1979-pmt1, Schoen:1989}. Instead of minimal hypersurfaces, we use marginally outer trapped hypersurfaces (MOTS) whose existence is guaranteed by earlier work of the first named author \cite{Eichmair:2009-Plateau}. An important part of our proof is to introduce an appropriate substitute for the area functional that is used in the time-symmetric case to single out certain minimal hypersurfaces. We also establish a density theorem of independent interest and use it to reduce the general case of the spacetime positive mass theorem to  the special case of initial data that has harmonic asymptotics and satisfies the strict dominant energy condition. 
\end{abstract}


\section{Introduction}

The following theorem is the main result of this paper. The technical terms are defined in Section \ref{section-definitions}. 

\begin{thm}[Spacetime positive mass theorem] \label{thm:PMT}
Let $3\le n<8$ and let $(M, g,k)$ be an $n$-dimensional asymptotically flat initial data set that satisfies the dominant energy condition.  Then
\[
E\ge |P|,
\]
where $(E,P)$ is the ADM energy-momentum vector of $(M,g,k)$.
\end{thm}

We briefly survey earlier results: The special case of Theorem \ref{thm:PMT} where $k\equiv0$ is called the time-symmetric case, or sometimes the Riemannian case. It is of particular importance. In the time-symmetric case, we have that $P=0$ and the dominant energy condition becomes the assumption that the scalar curvature of $g$ is nonnegative. The last named author and S.-T.\ Yau proved the time-symmetric case in dimension three in two articles from 1979 and 1981 \cite{Schoen-Yau:1979-pmt1, Schoen-Yau:1981-asymptotics}.  In \cite{Schoen-Yau:1979-pat}, they extended their proof of the time-symmetric case to dimensions less than $8$, as explained in detail in \cite{Schoen:1989}.   In 1981, they considered the general case $k\not \equiv0$ in dimension three and succeeded in proving that $E\ge0$ by solving Jang's equation  \cite{Schoen-Yau:1981-pmt2}.  Later, E. Witten discovered a completely different proof that $E\ge |P|$ in dimension three \cite{Witten:1981, Parker-Taubes:1982}.  Witten's technique easily generalizes to all higher dimensions, as long as the manifold is spin \cite{Bartnik:1986, Ding:2008}.  In dimensions higher than $7$, a complication arises in the Schoen-Yau argument due to possible singularities of minimal hypersurfaces.  Two different strategies for handling this complication have been announced by J.\ Lohkamp in a preprint \cite{Lohkamp:2006} from 2006 and by the last named author in 2009.  The first named author has generalized the spacetime $E\ge0$ theorem to dimensions less than $8$ (without spin assumption) in \cite{Eichmair:2008}.

For earlier history of this problem, we refer to the introduction of \cite{Schoen-Yau:1979-pmt1}.  The $E\ge0$ theorem is sometimes called the positive mass
theorem in the literature. We prefer to refer to it more accurately as the positive energy theorem. We reserve the phrase positive mass theorem for the $E \ge |P|$ theorem. This result could
also reasonably be called the
future timelike energy-momentum theorem.

Our proof of Theorem \ref{thm:PMT} is self-contained rather than by reduction to a previously known case. In particular, it gives a new proof of the $E \geq 0$ theorem for non-time-symmetric data. It follows from the work of D. Christodoulou and N. \'O Murchadha \cite{CO:1981} that the $E \geq0 $ theorem implies the $E  \geq |P|$ theorem in the vacuum case via a boost of the initial data slice in its spacetime development. At the end of this paper, we explain how our methods from Section \ref{section-density} allow for a generalization of this boost argument to arbitrary initial data satisfying the dominant energy condition. This provides an alternative proof of Theorem \ref{thm:PMT}. 

Our main theorem does not include a characterization of the equality case $E=|P|$.  The natural conjecture states that if $E=|P|$ in Theorem \ref{thm:PMT}, then $E=|P|=0$ and $(M,g)$ can be isometrically embedded into Minkowski space with second fundamental form $k$. Our proof of Theorem \ref{thm:PMT} is by contradiction, so the analysis of the equality case will require a substantial new idea. We note that the so-called equality case of the Riemannian positive mass theorem is derived from the nonnegativity of mass, but that its proof is unrelated to the proof of nonnegativity of mass. The situation in the case of general data is more complicated. It provides an interesting direction for future research.

The desired rigidity statement described in the preceding paragraph is already known to hold for spin manifolds. Although Witten sketched the basic idea for proving rigidity of spin manifolds in his 1981 article \cite{Witten:1981}, a complete, rigorous proof in all dimensions was not given until the work of P.T.\ Chru{\'s}ciel and D.\ Maerten in 2006 \cite{Chrusciel-Maerten:2006}. Their argument is based on R.\ Beig and P.T.\ Chru{\'s}ciel's 1996 proof in dimension three \cite{Beig-Chrusciel:1996}.

We briefly review the minimal hypersurface proof of the time-symmetric positive mass theorem in \cite{Schoen-Yau:1979-pmt1, Schoen:1989}. The argument is by induction on the dimension $3 \leq n < 8$. It proceeds by contradiction. Suppose that there exists an asymptotically flat Riemannian manifold  $(M, g)$ with nonnegative scalar curvature and negative mass $E<0$.  By a density argument  \cite{Schoen-Yau:1981-asymptotics}, one may assume that $(M,g)$ has harmonic asymptotics and positive scalar curvature.  The harmonic asymptotics and $E<0$ assumptions imply that  the coordinate planes $x^n=\pm\Lambda$ are barriers for minimal hypersurfaces for all sufficiently large $\Lambda$.  Consider an $(n-1)$-dimensional vertical cylinder $\partial C_\rho$ of large radius $\rho$ in the asymptotically flat coordinate chart.  For every $h\in[-\Lambda,\Lambda]$ there is an area-minimizing hypersurface $\Sigma_{\rho,h} \subset C_\rho$ with boundary equal to the height $h$ sphere on $\partial C_\rho$.  If $n<8$, this area-minimizing hypersurface is smooth.  Every such $\Sigma_{\rho,h}$ lies between the barrier planes $x^n=\pm\Lambda$. The area $\area(\Sigma_{\rho,h})$ is minimized over $h$ by some $h_\rho \in(-\Lambda,\Lambda)$. The corresponding surface $\Sigma_{\rho, h_\rho}$ has the property that
\begin{equation}\label{area-inequality}
\left.\frac{d^2}{dt^2}\right|_{t=0} \area(\Phi_t (\Sigma_{\rho,h_\rho}))\ge0
\end{equation}
for any variation $\Phi_t$ of  $\Sigma_{\rho,h_\rho}$ that is equal to vertical translation along $\partial C_\rho$.  One can then extract a smooth subsequential limit $\Sigma_{\infty}$ of $\Sigma_{\rho,h_{\rho}}$ as $\rho\to\infty$.   This $\Sigma_{\infty}$ is itself an $(n-1)$-dimensional asymptotically flat manifold with energy equal to zero.  Moreover, $\Sigma_{\infty}$ is a stable minimal hypersurface of $M$.  Owing to (\ref{area-inequality}), $\Sigma_\infty$ is stable with respect to variations that are (sufficiently close to) vertical translations outside a compact set. Using the well-established relationship between the stability of minimal hypersurfaces and scalar curvature, the stability of $\Sigma_\infty$ allows one to construct a conformal factor that changes the metric on $\Sigma_\infty$ to one with zero scalar curvature. The stability with respect to variations that are close to vertical translations is then used to show that the conformal factor must decrease the energy of $\Sigma_\infty$, thereby violating the time-symmetric positive mass theorem in dimension $n-1$.  For the base case of the induction, when $n=3$, one can show that the stability of $\Sigma_\infty$ and its asymptotics at infinity are incompatible with the Gauss-Bonnet Theorem.  When $n=3$, choosing a special height $h_\rho$ turns out to be unnecessary.

Our approach to the spacetime positive mass theorem is essentially a generalization of the proof described above.  In particular, it does not use the time-symmetric positive mass theorem as an input, as was done in the Jang equation approach of \cite{Schoen-Yau:1981-pmt2}.  
The proof is again by contradiction.  Let $(M,g,k)$ be an $n$-dimensional asymptotically flat initial data set satisfying the dominant energy condition $\mu \geq |J|$ and such that $E<|P|$. By our density theorem from Section \ref{section-density}, we may assume that $(M,g,k)$ has harmonic asymptotics and satisfies the strict dominant energy condition $\mu > |J|$. We may assume further that $P$ points in the vertical direction $-\partial_{n}$ of the asymptotically flat coordinate chart. The harmonic asymptotics and $E<|P|$ assumptions imply the coordinate planes $x^n=\pm\Lambda$ are barriers for marginally outer trapped hypersurfaces (MOTS) for all sufficiently large $\Lambda$.  Again, we consider an $(n-1)$-dimensional vertical cylinder $\partial C_\rho$ of large radius $\rho$. Let $h \in [- \Lambda, \Lambda]$. The results from \cite{Eichmair:2009-Plateau} guarantee the existence of a MOTS $\Sigma_{\rho,h}$ whose boundary is equal to the height $h$ sphere on $\partial C_\rho$.  This MOTS is smooth if $n<8$. Moreover, $\Sigma_{\rho, h}$ lies between the planes $x^n = \pm \Lambda$ and is stable in the sense of MOTS with boundary \cite{Galloway-Murchadha:2008}.  
Since MOTS are not known to arise from a variational principle, there is no canonical way of singling out a suitable height $h_\rho$ as in the time-symmetric case. To overcome this, we introduce a new functional $\F$ on hypersurfaces with boundary on $\partial C_\rho$ such that for some $h_\rho \in (-\Lambda, \Lambda)$ we (roughly) have that 
\begin{equation}\label{F-inequality}
\left.\frac{d}{dh}\right|_{h=h_\rho} \F(\Sigma_{\rho,h})\ge0. 
\end{equation}
Inequality (\ref{F-inequality}) in conjunction with the MOTS-stability of $\Sigma_{\rho, h}$ plays a role similar to that of (\ref{area-inequality}) in the time-symmetric case. Note that the $h_\rho$ selected in the time-symmetric case by minimization would satisfy  (\ref{F-inequality}) in our more general argument. As before, we extract a smooth subsequential limit $\Sigma_{\infty}$ of  $\Sigma_{\rho,h_\rho}$ as $\rho\to\infty$.  This $\Sigma_{\infty}$ is itself an $(n-1)$-dimensional asymptotically flat manifold with energy equal to zero, and $\Sigma_{\infty}$ is a stable MOTS in $M$. Using the relationship between stability of MOTS and scalar curvature established in \cite{Galloway-Schoen:2006}, one can construct a conformal factor that changes the metric on $\Sigma_\infty$ to one with zero scalar curvature.  Finally, (\ref{F-inequality}) plays the role of (\ref{area-inequality}) in establishing that the conformal factor must decrease the energy of $\Sigma_\infty$, thereby violating the time-symmetric positive mass theorem in dimension $n-1$.  

As in the time-symmetric case, the delicate height-picking argument is not required when $n=3$. 

The structure of the paper is as follows.  Section \ref{section-definitions} sets up the basic definitions and recalls some useful background material.  Section \ref{section-construction} establishes the existence of the MOTS needed for the proof.  Section \ref{dimension-three} completes the $n=3$ case of Theorem \ref{thm:PMT}.  The basic $n=3$ argument, which is explained in detail in Sections  \ref{section-construction} and \ref{dimension-three}, was first sketched out in 
 \cite[Section 7.2]{Schoen:2005}.  Section \ref{higher-dimensions} contains the parts of the proof that are specific to dimensions greater than three, including the height-picking procedure. In the last section, we show that an initial data set which satisfies the dominant energy condition can be perturbed by a small amount to one with harmonic asymptotics that satisfies the strict dominant energy condition. 
 

\section{Definitions, notation, and basic facts}\label{section-definitions}

\begin{defin} Let $B$ be a closed ball in $\rr^n$ with center at the origin. For every $k\in\{0, 1, \ldots\}$, $p \geq 1$, and $q\in\rr$ we define the \emph{weighted Sobolev space} $W^{k,p}_{-q}(\rr^n\smallsetminus B)$ as the collection of those  $f \in W^{k, p}_{loc} ( \rr^n\smallsetminus B)$  with 
\[
		\| f\|_{W^{k,p}_{-q}(\rr^n\smallsetminus B)} := \left( \int_{\rr^n\smallsetminus B} \sum_{|I| \le k} \left( \big|( \partial^{I} f )(x) \big|  |x|^{ | I| + q }\right)^p |x|^{-n} \, d x  \right)^{\frac{1}{p}}<\infty. 
\]
We usually write $L^p_{-q}$ instead of $W^{0, p}_{-q}$. 

Suppose now that  $M$ is a $\C^k$ manifold such that there is a compact set $K\subset M$ and a diffeomorphism $M\smallsetminus K \isom  \rr^n\smallsetminus B$. The $W^{k,p}_{-q}$ norm on $M$ is defined in a routine way by choosing an atlas for $M$ that consists of the diffeomorphism $M\smallsetminus K \isom  \rr^n\smallsetminus B$ and finitely many precompact charts, and then summing the $W^{k,p}_{-q}(\rr^n\smallsetminus B)$ norm on the noncompact chart and the $W^{k,p}$ norms on the precompact charts.  The resulting space $W^{k,p}
_{-q}(M)$ and its topology only depend on the diffeomorphism $M\smallsetminus K \isom  \rr^n\smallsetminus B$.  This definition can be extended to the tensor bundles of $M$ simply by considering components with respect to these charts.  We usually write $W^{k,p}_{-q}$ for $W^{k,p}_{-q}(M)$ when the context is clear.
\end{defin}

\begin{defin}  Let $B$ be a closed ball in $\rr^n$ with center at the origin.
For every $k\in\{0, 1, \ldots\}$, $\alpha\in(0,1)$, and $q\in\rr$ we define the \emph{weighted H\"{o}lder space} $\C^{k,\alpha}_{-q}(\rr^n\smallsetminus B)$ as the collection of those $f \in C^{k, \alpha}_{loc} (\rr^n\smallsetminus B)$ with
\[
		\| f\|_{\C^{k,\alpha}_{-q}(\rr^n\smallsetminus B)} := \sum_{|I| \le k} \sup_x \left| |x|^{ | I| + q }(\partial^{I} f)(x) \right| +  \sum_{|I| = k} \left[ |x|^{\alpha + | I| + q }(\partial^{I} f)(x) \right]_\alpha
		<\infty. 
\]
Suppose now that  $M$ is a $\C^k$ manifold such that there is a compact set $K\subset M$ and a diffeomorphism $M\smallsetminus K \isom  \rr^n\smallsetminus B$. The space $\C^{k,\alpha}_{-q}(M)$ can then be defined just as we did for $W^{k,p}_{-q}(M)$  in the preceding definition.
\end{defin}

\begin{defin} 
Let $n\ge3$.
An \emph{initial data set} is an $n$-dimensional manifold $M$ equipped with a complete $\C^{2}_{loc}$ Riemannian metric $g$ and a $\C^{1}_{loc}$ symmetric $(0,2)$-tensor $k$. On an initial data set, one can define the \emph{mass density} $\mu$ and the \emph{current density} $J$ by
\begin{align*} 
\begin{split}
	\mu&=\tfrac{1}{2}\left(R_g  - | k |_g^2 + (\tr_g k )^2\right)\\
	J&=  \Div_g k - d(\tr_g k).
\end{split}
\end{align*}
We say that $(M,g,k)$ satisfies the \emph{dominant energy condition} if
\[ \mu\ge |J|_g.\]
  It is often convenient to consider the \emph{momentum tensor} 
\[\pi= k -(\tr_g k)g.\]
It contains the same information as $k$ since $k = \pi - \tfrac{1}{n-1}(\tr_g \pi)g$.

Let 
\[ p > n,  \quad q \in ((n-2)/2, n-2),  \quad q_0>0,  \quad \text{and} \quad \alpha \in (0, 1 - n/p].\]  We say that an initial data set $(M,g, k)$ is \emph{asymptotically flat}\footnote{There are several incompatible notions of  \emph{asymptotic flatness} in the literature.} of type $(p,q,q_0,\alpha)$ if $g \in C^{2, \alpha}_{loc} (M)$, $k \in C^{1, \alpha}_{loc}(M)$, and if there is a compact set $K\subset M$ and a diffeomorphism $M\smallsetminus K\isom \rr^n\smallsetminus B$ for some closed ball $B\subset \rr^n$ such that
\[
(g-\delta,k)\in W^{2,p}_{-q}(M)\times W^{1,p}_{-1-q}(M)
\]
where $\delta$ is a smooth symmetric $(0,2)$-tensor that coincides with the Euclidean inner product on $M\smallsetminus K\isom \rr^n\smallsetminus B$, and such that 
\[ 
\mu, J \in\C^{0,\alpha}_{-n-q_0} (M).
\]
If $(M, g, k)$ is asymptotically flat, one can define the \emph{ADM energy} $E$ and the \emph{ADM momentum} $P$ as
\begin{align*}
E &= \tfrac{1}{2(n-1)\omega_{n-1}}  \lim_{r\to\infty}\int_{|x|=r}\sum_{i,j=1}^n (g_{ij,i}-g_{ii,j})\nu_0^j\, d\area_0\\
P_i&=\tfrac{1}{(n-1)\omega_{n-1}} \lim_{r\to\infty}  \int_{|x|=r}\sum_{j=1}^n \pi_{ij}\nu_0^j\, d\area_0\qquad i= 1, 2, \ldots, n. 
\end{align*}
Here, the integrals are computed in the coordinate chart $M \setminus K \cong_x \rr^n\smallsetminus B$, $\nu_0^j= x^j / |x|$, $\area_0$ is the $(n-1)$-dimensional Euclidean Hausdorff measure, and $\omega_{n-1}$ is the volume of the standard unit sphere in $\rr^n$. 
\end{defin} 

\begin{remark}
Theorem \ref{thm:PMT} still holds if we allow multiple asymptotically flat ends in the definition of initial data sets. We simply use the large celestial spheres in the other ends as barriers in the proof of Lemma \ref{MOTS-existence}. 
\end{remark} 

In order to carry out our main argument, we require better asymptotic behavior.
\begin{defin} 
Let $n\geq3$ and let $(M, g, k)$ be an $n$-dimensional asymptotically flat initial data set .  We say that $(M,g,k)$ has \emph{harmonic asymptotics} if there exists a $\C^{3,\alpha}$ diffeomorphism as in the definition of asymptotic flatness, as well as a $\C^{2,\alpha}$ function $u$ and a $\C^{2,\alpha}$ vector field $Y$ such that for $i,j= 1, 2, \ldots, n$,
\begin{align*}
u(x)& = 1+ a|x|^{2-n}+ O^{2, \alpha}(|x|^{1-n})\\
Y_i(x)&=b_i|x|^{2-n}+O^{2, \alpha}(|x|^{1-n})\\
g_{ij}&=u^{\frac{4}{n-2}}\delta_{ij}\\
\pi_{ij}& = u^{\frac{2}{n-2}}\left[(L_Y \delta)_{ij} -(\Div_\delta Y)\delta_{ij}\right]
\end{align*}
where $a$, $b_1, \ldots, b_n$ are constants, $\delta_{ij}$ is the Euclidean metric, and $L_Y$ is the Lie derivative.  Here and below, an expression $O^{k, \alpha}(|x|^{-q})$ stands for a function in the weighted H\"{o}lder space $\C^{k,\alpha}_{-q}$.  
\end{defin}

\begin{notation} Let $(M, g, k)$ be an $n$-dimensional initial data set. Let $\Sigma$ be a two-sided $\C^{3,\alpha}$ hypersurface with boundary in $M$ with unit normal $\nu$.  Let $D$ denote the ambient covariant derivative.
We define the second fundamental form $B_\Sigma$ and shape operator $S_\Sigma$ of $\Sigma$ using the convention
\[ B_\Sigma(X,Y)= \langle S_\Sigma (X),Y\rangle= \langle D_X \nu, Y\rangle\]
for vector fields $X,Y$ tangent to $\Sigma$, where the angle brackets denote the inner product $g$.  We define the mean curvature scalar $H_\Sigma$ to be the trace of $S_\Sigma$. According to this convention, the mean curvature of a sphere in $\rr^n$ with respect to the outward pointing unit normal is positive. We also define the \emph{expansion} 
\[\theta^+_{\Sigma} = H_\Sigma + \tr_\Sigma k\]
of $\Sigma$, where $\tr_\Sigma k$ denotes the trace over the tangent space of $\Sigma$.
If $\theta^+_\Sigma$ vanishes on all of $\Sigma$, we say that $\Sigma$ is a \emph{marginally outer trapped hypersurface}, or \emph{MOTS} for short. 

Note that the property of being a MOTS depends on the choice of normal.  For a vector field $X$ defined along $\Sigma$ but not necessarily tangent to it, we let
\[\Div_\Sigma X\]
be the function on $\Sigma$ which at $x \in \Sigma$ equals $\sum_{i=1}^{n-1} \langle D_{e_i} X, e_i\rangle$ where $e_1,\ldots,e_{n-1}$ is an orthonormal basis of $T_x \Sigma$. 
\end{notation}

\begin{notation}
Given a vector field $X$ defined along $\Sigma$, we can decompose $X$ into its normal and tangential components
\[X=\varphi \nu +\hat{X}.\]
Throughout this paper, whenever there is a vector field with the variable name $X$ on a hypersurface $\Sigma$, the function $\varphi$ and the tangent field $\hat{X}$ are defined this way. We use $\eta$ to denote the outward pointing unit normal of $\partial\Sigma$ in $\Sigma$.
\end{notation}

The expression for the linearization of the expansion stated in the following proposition generalizes the well-known formula for the variation of the mean curvature.  See, for example, \cite{Galloway-Schoen:2006}. 
\begin{prop}\label{prop:Dtheta}
Let $\Sigma$ be a two-sided hypersurface with boundary in an $n$-dimensional initial data set $(M, g, k)$, and let $\nu$ be a continuous unit normal field along $\Sigma$.  Let $X \in \mathfrak{X}(M)$ be a $\C^2$ vector field, and let $\Phi_t$ be the flow generated by $X$.  We can compute the expansion $\theta^+_{\Sigma_t}$ of the push forward $\Sigma_t := \Phi_t (\Sigma)$ with respect to the unit normal that points in the direction of $\Phi_{t*} (\nu)$ and pull it back to a function on $\Sigma$ using $\Phi_t$. The derivative of this function in $t$ at $t=0$ is denoted by $D \theta^+|_{\Sigma}(X)$. We have that 
\begin{equation}\label{dtheta}
D\theta^+|_\Sigma(X)= -\Delta_\Sigma \varphi+2 \langle W_\Sigma, \nabla\varphi\rangle  + (\Div_\Sigma W_\Sigma - |W_\Sigma|^2+Q_\Sigma)\varphi+ \nabla_{\hat{X}}\theta^+_\Sigma,
\end{equation}
where
\[Q_\Sigma=\tfrac{1}{2}R_\Sigma-\mu-J(\nu)-\tfrac{1}{2}|k_\Sigma+B_\Sigma|^2.\]
Here, $k_\Sigma$ denotes the restriction of $k$ to vectors tangent to $\Sigma$, and $W_\Sigma$ is the tangential vector field on $\Sigma$ that is dual to the $1$-form $k(\nu,\cdot)$ along $\Sigma$.

 \end{prop}

\begin{notation}
Throughout this paper, we will drop the $\Sigma$ subscripts when the context is clear.  Everything is computed with respect to the metric $g$ unless noted otherwise. In particular, we use $\mathcal{H}$ to denote the Hausdorff measures associated with $g$. 
\end{notation}

\begin{defin}
Let $\Sigma$ be a MOTS in an initial data set $(M,g,k)$.  We define the operator  
\begin{equation}\label{define-L}
L_\Sigma v :=-\Delta_\Sigma v +2 \langle W_\Sigma, \nabla v\rangle  + (\Div_\Sigma W_\Sigma - |W_\Sigma|^2+Q_\Sigma)v,
\end{equation}
where $v$ is a function on $\Sigma$.  Although this operator is not self-adjoint, the Krein-Rutman Theorem shows that there is a unique (Dirichlet) eigenvalue with least real part. It is called the \emph{principal (Dirichlet) eigenvalue} of $L_\Sigma$.  This eigenvalue is real. If $\Sigma$ is connected, the corresponding eigenspace is one-dimensional  and generated by a $\C^{2,\alpha}$ \emph{principal eigenfunction} that is positive on the interior of $\Sigma$ \cite[p.\ 3]{Galloway-Murchadha:2008}.  If the principal eigenvalue is nonnegative, we say that  $\Sigma$ is a \emph{stable} MOTS.  This concept of stability was introduced in \cite{Galloway-Murchadha:2008}, based on the analogous definition for closed MOTS in \cite{Andersson-Mars-Simon:2005}. It is easy to see that this generalizes the notion of stability of minimal hypersurfaces with boundary. 
\end{defin}

\begin{prop} \label{symmetrized-stability} 
Let $\Sigma$ be a stable MOTS in an initial data set $(M,g,k)$. For every compactly supported $\C^1$ function $v$ on $\Sigma$ that vanishes along $\partial\Sigma$, we have that
\begin{equation}
\int_{\Sigma} (|\nabla v|^2  + Q_\Sigma v^2 )\,d\area \geq 0.
\end{equation}
\end{prop}
This follows from an argument in \cite{Galloway-Schoen:2006}, \emph{cf.\ }the proof of Lemma \ref{stability} below.

We state the following geometric variant of the strong maximum principle for ordered hypersurfaces that are subsolutions and supersolutions of the same prescribed mean curvature equation. We refer to \cite[Lemma 1]{Schoen:1983}, \cite[Proposition 3.1]{Andersson-Metzger:2009}, and \cite[Proposition 2.4]{Ashtekar-Galloway:2005} for similar results and partial proofs. It is important to pay attention to the choice of normal here. A good example to have in mind is the following: Consider a sphere of radius $2$ that is tangent to a sphere of radius $1$ in Euclidean space.  Either the smaller one is enclosed by the larger one or it lies outside of it.  The (obvious) conclusion of the lemma in this simple example is that the larger sphere cannot lie inside the smaller one.

\begin{prop}[Strong maximum principle] \label{max-prin}
Let $g$ be a $\C^2$ Riemannian metric on $M = \bar B^{n-1}_1 (0) \times (-2, 2) \subset \mathbb{R}^n$. Let $F$ be a $\C^1$ function on the unit sphere bundle of $M$ and let $u_1, u_2 \in \C^2(\bar B^{n-1}_1(0))$ be such that $-1 \leq u_1(x') \leq u_2(x') \leq 1$ for all $x' \in \bar B^{n-1}_1(0)$. Assume that the hypersurfaces with boundary $\Sigma_i = \textrm{graph} (u_i) \subset M$ are such that $H_{\Sigma_1} (x) \leq F(x, \nu_{\Sigma_1}(x))$ for all $x \in \Sigma_1$ and $H_{\Sigma_2} (x) \geq F(x, \nu_{\Sigma_2}(x))$ for all $x \in \Sigma_2$ where the mean curvatures are computed using the upward pointing unit normals.  If $\Sigma_1$ and $\Sigma_2$ intersect at an interior point or are tangent to each other at a boundary point, then they must be equal.
\end{prop}

Let $g_1, g_2$ be two metrics on an $n$-dimensional manifold $M$ that are related by $$g_2=u^{\frac{4}{n-2}}g_1.$$ The scalar curvatures of these metrics are related by
\begin{equation}\label{scalar-curvature}
R_2 = u^{-\frac{n+2}{n-2}}\left(R_1 u-\tfrac{4(n-1)}{n-2}     \Delta_1 u \right).
\end{equation}
Let $\Sigma \subset M$ be a two-sided hypersurface. If $\nu_1$ is a unit normal with respect to $g_1$, then $\nu_2=u^{\frac{-2}{n-2}}\nu_1$ is a unit normal with respect to $g_2$. The corresponding mean curvatures are related by 
\begin{equation}\label{mean-curvature}
H_2 = u^{\frac{-2}{n-2}}\left(H_1 +\tfrac{2(n-1)}{n-2}u^{-1}\nabla_{\nu_1}u\right).
\end{equation}
If $(M, g_1)$ is an asymptotically flat $n$-dimensional Riemannian manifold and if $u=1+a|x|^{2-n}+O^{2, \alpha}(|x|^{1-n})$ is $\C^{2,\alpha}$, then $(M,g_2)$ is also asymptotically flat. The energies of $(M, g_1)$ and $(M, g_2)$ are related by the formula
\begin{align}
E_2 &= E_1 -  \tfrac{2}{(n-2)\omega_{n-1}} \lim_{r\to\infty} \int_{|x|=r}u \nabla_{\nu_1} u \,d\area_1\label{mass-change-u}\\
&=E_1+2a.\label{mass-change}
\end{align}

The proof of the positive mass theorem in the time-symmetric case uses regularity and compactness properties of area minimizing hypersurfaces. By contrast, MOTS are not known to obey a useful variational principle. We will use the theory of almost minimizing currents as a viable substitute in our proof of Theorem \ref{thm:PMT}. 

\begin{defin} [\cite{Duzaar-Steffen:1993}] 
Let $(M,g)$ be a complete $n$-dimensional Riemannian manifold and let $T$ be an integral $k$-current in $M$. 
Let $U\subset M$ be an open set such that $\spt (\partial T) \cap U = \emptyset$.
Then $T$ is $\lambda$-\emph{minimizing} in $U$ if for every integral $(k+1)$-current $X$ with support in $U$ we have that
\[\mathbf{M}_U(T) \le \mathbf{M}_U(T+\partial X) + \lambda\mathbf{M}_U(X).\]
Here, $\mathbf{M}_U$ denotes the mass of a current in $U$.  
\end{defin}
This particular almost minimizing property was introduced and studied systematically by F. Duzaar and K. Steffen in \cite{Duzaar-Steffen:1993}. 
In \cite{Eichmair:2009-Plateau}, the first named author of the present article observed that the
$\lambda$-minimizing property is a natural feature of the MOTS that arise in the existence theory of the Plateau problem developed in \cite{Eichmair:2009-Plateau}, despite the absence of a useful variational principle.
The  properties of $\lambda$-minimizing currents  that we use in the proof of Theorem \ref{thm:PMT} below  are summarized in \cite[Appendix A]{Eichmair:2009-Plateau}.  


\section{Construction of MOTS}\label{section-construction}

Our proof of Theorem \ref{thm:PMT} will be by induction on dimension and contradiction.  Let $3\le n<8$, and suppose there exists an $n$-dimensional asymptotically flat initial data set $(M, g, k)$ of type $(p,q,q_0,\alpha)$ satisfying the dominant energy condition, but $E<|P|$.   For the case $n=3$, we will obtain a contradiction to the Gauss-Bonnet Theorem in Section \ref{dimension-three}, and for $3<n<8$, we will obtain a contradiction to the time-symmetric case of Theorem \ref{thm:PMT} in dimension $n-1$.  By the density theorem (Theorem \ref{th:density-theorem-section}), we can assume without loss of generality, that $(g,k)$ has harmonic asymptotics and satisfies the strict dominant energy condition $\mu>|J|_g$.  Specifically, we can choose asymptotically flat coordinates $(x^1, \ldots, x^n)$
on $M\smallsetminus K \cong_{x} \mathbb{R}^n \smallsetminus B$ for some closed ball $B$ such that on  $\mathbb{R}^n \smallsetminus B$, we have, for $i,j= 1, 2, \ldots, n$,  
\begin{align*}
g_{ij}&=u^{\frac{4}{n-2}}\delta_{ij}\\
\pi_{ij}& = u^{\frac{2}{n-2}}\left[(L_Y \delta)_{ij} -(\Div_\delta Y)\delta_{ij}\right]
\end{align*}
for some $u, Y\in \C^{2,\alpha}_{2-n}$ satisfying
\begin{align*}
u(x)& = 1+ a|x|^{2-n}+ O^{2, \alpha}(|x|^{1-n})\\
Y_i(x)&=b_i|x|^{2-n}+O^{2, \alpha}(|x|^{1-n}).
\end{align*}
Without loss of generality, we assume that $P = (0, \ldots, 0, - |P|)$. 

\subsection{Existence of horizontal barriers}

\begin{lem}\label{barriers}
Let $(M,g,k)$ be as described above. 
Then, for sufficiently large $\Lambda$, we have that $\theta^+_{\{x^n = \Lambda\}}>0$ and $\theta^+_{\{x^n = -\Lambda\}}<0$ where the expansion is computed with respect to the upward pointing unit normal. 
\end{lem}

\begin{proof}
It follows from (\ref {mass-change}) that the $|x|^{2-n}$ coefficient of the function $u$ is just $a=\frac{E}{2}$. We claim that the $|x|^{2-n}$ coefficient of $Y_i$ is $b_i=-\tfrac{n-1}{n-2}P_i$, i.e. that $b_n=\tfrac{n-1}{n-2}|P|$ and $b_i=0$ for $i<n$.  To see this, note that
\begin{align*}
(n-1)\omega_{n-1}P_i
&=  \lim_{r\to\infty}  \int_{|x|=r} \pi_{ij}\nu_0^j\, d\area_0 \\
&= \lim_{r\to\infty} \int_{|x|=r} \left(Y_{i,j}+Y_{j,i} -(\Div_\delta Y)\delta_{ij}\right)\nu_0^j\, d\area_0 \\
&= \lim_{r\to\infty}  \int_{|x|=r} (2-n) |x|^{1-n}\left[ b_i(\nu_0)_j + b_j (\nu_0)_i -  b_k \nu_0^k \delta_{ij}+O(|x|^{-1})\right]\nu_0^j\,d\area_0\\
&= -(n-2)b_i\omega_{n-1}. 
\end{align*}

\noindent We claim that
\begin{equation}\label{plane-theta}
\theta^+_{\{x^n = \Lambda\}} =  H_{\{x^n = \Lambda\}} + \tr_{\{x^n = \Lambda\}} (k)= (n-1)(|P|-E)\Lambda|x|^{-n} +O(|x|^{-n}).
\end{equation}
To see this, we use harmonic asymptotics and formula (\ref{mean-curvature}) to compute
\begin{align*}
H_{\{x^n = \Lambda\}}
& = \tfrac{2(n-1)}{n-2}u^{\frac{-2}{n-2}-1} \partial_n u\\
& =  \tfrac{2(n-1)}{n-2}u^{\frac{-2}{n-2}-1}  [ (2-n)a |x|^{-n} x^n+O(|x|^{-n})]\\
& =  -2(n-1)a |x|^{-n}x^n+O(|x|^{-n})\\
& =   -(n-1)E |x|^{-n}\Lambda+O(|x|^{-n}).
\end{align*}
To compute $\tr_{\{x^n = \Lambda\}} (k)$, first note that 
\[ k_{ij}= u^{\frac{2}{n-2}} \left[(L_Y \delta)_{ij} -\tfrac{1}{n-1}(\Div_\delta Y)\delta_{ij}\right].\]
So we have
\begin{align*}
\tr_{\{x^n = \Lambda\}} (k) & = \sum_{i,j=1}^{n-1} g^{ij}k_{ij}\\
&= \sum_{i,j=1}^{n-1} u^{\frac{-2}{n-2}}\delta^{ij}[ Y_{i,j}+Y_{j,i} -\tfrac{1}{n-1}(\Div_\delta Y)\delta_{ij}]\\
&= \sum_{i,j=1}^{n-1} u^{\frac{-2}{n-2}}\delta^{ij}\left[ \tfrac{-1}{n-1}Y_{n,n}\delta_{ij}+O(|x|^{-n})\right]\text{ since }Y_i=O^{2, \alpha}(|x|^{1-n})\text{ for }i<n\\
&=  -Y_{n,n}+O(|x|^{-n})\\
&= (n-2)b_n|x|^{-n}x^n+O(|x|^{-n})\\
&= (n-1)|P| |x|^{-n}\Lambda+O(|x|^{-n}),
\end{align*}
completing the proof of (\ref{plane-theta}).  Note that (\ref{plane-theta}) shows that that for large enough $\Lambda$ one has $\theta^+_{\{x^n = \Lambda\}}>0$.  The proof that $\theta^+ _{\{x^n = - \Lambda\}} < 0$ is similar.
\end{proof}


\subsection{Existence of MOTS} \label{sec:existenceMOTS}

\begin{notation}
We now fix $\Lambda$ large enough so that Lemma \ref{barriers} applies. 
For large $\rho$, we define $C_\rho := K \cup x^{-1}\{(x^1, \ldots, x^{n-1}, x^n)\,|\, (x^1)^2 + \ldots + (x^{n-1})^2  < \rho^2\}$ to be the region horizontally bounded by a vertical coordinate cylinder of radius $\rho$, and we define $C_{\rho,h}$ to be the part of $C_\rho$ lying between the planes $x^n=\pm h$.  Define $\Gamma_{\rho, h} := \partial C_\rho \cap \{ x^n = h\}$.
 \end{notation}

\begin{lem} \label{MOTS-existence}
Let $(M,g,k)$ and $\Lambda$ be as described above. For every sufficiently large $\rho$ and all $h \in [- \Lambda, \Lambda]$ there exists a stable $\C^{3,\alpha}$ MOTS $\Sigma_{\rho,h}$ in $C_{\rho,\Lambda}$ whose boundary equals $\Gamma_{\rho,h}$ and which meets $\partial C_\rho$ transversely. Every $\Sigma_{\rho,h}$ is a $\lambda$-minimizing boundary in $C_{\rho,2\Lambda}$ where $\lambda$ depends only on $|k|_{\C^0}$. Moreover, we can choose $\{\Sigma_{\rho, h}\}_{|h|\leq\Lambda}$ so that $\Sigma_{\rho, h_1}$ lies strictly below $\Sigma_{\rho, h_2}$ as a boundary in $C_{\rho,\Lambda}$ if $- \Lambda \leq h_1 < h_2 \leq \Lambda$.
\end{lem}

\begin{remark}
The regularity of the MOTS $\Sigma_{\rho, h}$ is the only place in the proof of Theorem \ref{thm:PMT} where the assumption $n<8$ is used.  For $n>8$, the lemma still holds except that  the $\lambda$-minimizing boundaries $\Sigma_{\rho, h}$ are only regular away from a thin singular set. \end{remark}

\begin{figure*}[top]
\begin{picture}(0, 195)(-25, 0)
			\includegraphics[scale=0.65]{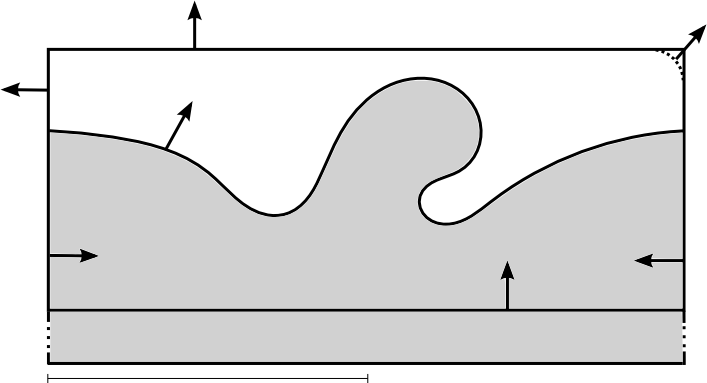}

			\put(-255 , 112){$\theta_{\Sigma_{\rho, h}}^+ = 0$}
			\put(-287 , 133){$\nu$}
			\put(-160 , 140){\Large$\Sigma_{\rho, h}$}
			\put(-10, 120){ \Large$\Gamma_{\rho, h} $}
			\put(-379, 120){ \Large$\Gamma_{\rho, h} $ }
			\put(-270,60){\Large$\Omega_{\rho,h}$}
			\put(-376,42){\Large$\partial C_{\rho}$}
			\put(-264,-5){$\rho$}

			\put(-375,162){\scriptsize$\theta^+ > 0$}
			\put(-260,185){\scriptsize$\theta^+ > 0$}
			\put(-40,185){\scriptsize$H \gg 1$}

			\put(-43,72){\scriptsize$\theta^+ < 0$}
			\put(-96,48){\scriptsize$\theta^+ < 0$}
			\put(-340,75){\scriptsize$\theta^+ < 0$}

			\put(-198, 181){$\{ x^n = \Lambda\}$}
			\put(-200, 47){$\{ x^n = - \Lambda\}$}
			\put(-200, 17){$\{ x^n = - 2\Lambda\}$}
\end{picture}
\caption{The expansion $\theta^+$ is computed with respect to the indicated unit normals.}
\end{figure*}

\begin{proof} 
First observe that by the decay conditions on $g$ and $k$, the coordinate cylinder $\partial C_\rho$ has $\theta^+>0$ with respect to the outward normal and $\theta^+<0$ with respect to the inward normal.

We would like to solve the Plateau problem for MOTS with boundary $\Gamma_{\rho,h}$ for each $h\in[-\Lambda,\Lambda]$.  Note that $\Gamma_{\rho,h}$ divides $\partial C_{\rho,2\Lambda}$ into a top piece and a bottom piece.  According to \cite{Eichmair:2010}, the MOTS Plateau problem is solvable if the top piece has $\theta^+>0$ with respect to the outward normal of $\partial C_{\rho,2\Lambda}$ and the bottom piece has $\theta^+<0$ with respect to the inward normal of $\partial C_{\rho,2\Lambda}$. By the observation above and Lemma \ref{barriers}, the two pieces of $\partial C_{\rho,2\Lambda}$ satisfy the desired trapping conditions, with the exception of the corners where $\partial C_\rho$ intersects $\{x^n=\pm2\Lambda\}$.  See Figure 1.  Intuitively, the corners do not cause a problem because the (singular) distributional mean curvature there has a favorable sign.  Therefore one could ``round-off'' the corners as in Figure 1.  Alternatively, we observe that the proof in \cite{Eichmair:2009-Plateau} can easily accommodate the corners:  Just as in that proof, we use the trapping of the cylindrical and horizontal pieces of $\partial C_{\rho,2\Lambda}$ to construct barriers that have appropriate blow-up behavior for Jang's equation, and then  combine these barriers as in \cite[Lemma 4.1]{Eichmair:2010}.  
Specifically, there exists a $\C^{3,\alpha}$ family of MOTS $\{\Sigma_{\rho, h}\}_{|h| \leq \Lambda}$ with $\partial \Sigma_{\rho, h} = \Gamma_{\rho, h}$ such that each $\Sigma_{\rho, h}$ is a $\lambda$-minimizing boundary in $C_{\rho,2\Lambda}$ where $\lambda = \lambda(|k|_{\C^{0}})$. 

Moreover, it can be seen from the construction in \cite{Eichmair:2009-Plateau} that the regions $\Omega_{\rho, h}\subset C_{\rho,2\Lambda}$ bounded by the $\Sigma_{\rho, h}$'s can be taken to be ordered, so that $\Omega_{\rho, h_1} \subset \Omega_{\rho, h_2}$ whenever $- \Lambda \leq h_1 \leq h_2 \leq \Lambda$.  To see this, note that the supersolutions $\overline{u}_{t, \rho, h}$ used in the construction of solutions to the regularized Jang's equation that lead to blow up in \cite[Lemma 4.1, bottom of p. 570]{Eichmair:2009-Plateau} can be taken to be pointwise decreasing in the parameter $h$, so that the corresponding Perron solutions $\overline{u}^P_{t, \rho, h}$ are decreasing in $h$ and hence their epigraphs are increasing. It has been shown in \cite{Eichmair-Metzger:2011c} that we may assume further that $\Sigma_{\rho, h}$ is stable in the sense of MOTS. 

Standard barriers for $\Sigma_{\rho, h}$ can be constructed from the trapped boundary $\partial C_\rho$ by slight ($\C^2$)-inward perturbation above respectively below its boundary $\Gamma_{\rho, h}$, locally uniformly in $(\rho, h)$, so that the angle (in the underlying Euclidean coordinate system) at which $\Sigma_{\rho, h}$ meets $\partial C_\rho$ is bounded away from $0$. For the special case of MOTS this inwards bending is explained in some detail in \cite[Section 3]{Eichmair:2009-Plateau}. Together with the $\lambda$-minimizing property and Allard's boundary regularity theorem \cite{Allard:1975} it follows that near $\partial C_\rho$ the hypersurface $\Sigma_{\rho, h}$ can be written as a vertical $\C^{1, \alpha}$ graph above $\{x^n = h\}$. In particular, we see that $\Sigma_{\rho,h}$ intersects $C_\rho$ transversely. Higher regularity of the defining function---which solves a prescribed mean curvature equation---then follows from Schauder theory in a standard way.

Finally, since each horizontal plane above $x^n=\Lambda$ has $\theta^+>0$ and each horizontal plane below $x^n = - \Lambda$ has $\theta^+ < 0$ by Lemma \ref{barriers} it follows from Lemma \ref{max-prin} that each $\Sigma_{\rho,h}$ actually lies in $C_{\rho,\Lambda}$. 
\end{proof}

\begin{remark} We are grateful to Brian White for helping us clarify the
following issue: When $M$ is an $n$-dimensional $\C^{1, \alpha}$
manifold and $g$ is a complete $\C^{\alpha}$ Riemannian metric on
it, then the standard interior Allard-type $\C^{1, \alpha}$
regularity of (almost) minimizing boundaries away from a set of Hausdorff
dimension at most $8$ holds. This was shown by J.\ Taylor in
\cite{Taylor:1976} (this part of the discussion in her paper applies to
$n$--dimensional manifolds). When the manifold is $\C^2$ and the
metric is Lipschitz, then this follows also from the work of R.\ Schoen and
L.\ Simon \cite{Schoen-Simon:1982} (for almost minimizers this was pointed
out by B. White in \cite[p. 498]{White:1991}). When the manifold is
$\C^4$ and the metric $\C^3$ so that the Nash embedding
theorem provides an isometric embedding of $(M, g)$ into a high dimensional
Euclidean space, then this also follows directly upon applying the
Euclidean regularity theory as in \cite{GMT}.  In the preceding
lemma, note that once we know that our surfaces are $\C^{1,
\alpha}$, we can then apply Schauder theory to the MOTS equation to obtain $\C^{3, \alpha}$ regularity.  The boundary regularity follows more easily
because the metric is conformal to the Euclidean metric there. 
\end{remark}

\subsection{Convergence of MOTS}\label{sec:convergenceMOTS}

Although there is no reason to expect the family $\{\Sigma_{\rho, h}\}_{|h| \leq \Lambda}$ to form a $\C^{3,\alpha}$ foliation (even when there are no topological obstructions), we can still prove a partial regularity result.
Its proof is similar to the proof of regularity of the outermost MOTS (established in \cite{Andersson-Metzger:2009} for $n=3$ and then in \cite{Eichmair:2010} for $3 \leq n <7$) but simpler, because the two-sided $\lambda$-minimizing property ensures embeddedness. 

\begin{lem}\label{MOTS-convergence} Let $\{\Sigma_{\rho, h}\}_{|h| \leq \Lambda}$ be as in Lemma \ref{MOTS-existence}. For each $h_0\in(-\Lambda,\Lambda]$, the upper envelope of $\{ \Sigma_{\rho, h}\}_{h<h_0}$ is a $\C^{3,\alpha}$ MOTS with boundary $\Gamma_{\rho, h_0}$ which we denote by $\underline{\Sigma}_{\rho, h_0}$.  By convention we define $\underline{\Sigma}_{\rho, -\Lambda}:=\Sigma_{\rho, -\Lambda}$.  Moreover, $\lim_{h \nearrow h_0}\Sigma_{\rho, h}=\underline{\Sigma}_{\rho, h}$ in the $\C^{3,\alpha}$ topology.  We define $\overline{\Sigma}_{\rho, h_0}$ as the lower envelope of $\{\Sigma_{\rho, h}\}_{h>h_0}$ for $h_0 \in [-\Lambda, \Lambda)$ and $\overline{\Sigma}_{\rho, \Lambda}:= \Sigma_{\rho, \Lambda}$ and note that analogous statements hold for these hypersurfaces.
\end{lem}
\begin{proof}
Fix $h_0\in(\Lambda,\Lambda]$. Let $- \Lambda \leq h_i \nearrow h_0$ as $i \to \infty$ and pass the $\lambda$-minimizing boundaries $\Sigma_{\rho, h_i}$ to a subsequential limit $\underline{\Sigma}_{\rho, h_0}$. Current convergence is automatic from the mass bounds, varifold convergence follows because there is no mass loss in limits of $\lambda$-minimizing currents, and $\C^{3,\alpha}$ convergence follows from Allard's interior and boundary regularity theorems. Note that since the $\Sigma_{\rho, h_i}$ are increasing, this limit is independent of the choice of subsequence, and hence it is really a limit of the original sequence $h_i$.  The limit does not depend on the choice of sequence $h_i$ for the same reason. 
\end{proof}

\begin{defin} \label{defn:jumpheight} We say that $h_0 \in [- \Lambda, \Lambda]$ is a \emph{jump height} if $\overline{\Sigma}_{\rho, h_0}$ does not equal $\underline{\Sigma}_{\rho, h_0}$. \end{defin}

By the previous lemma, $\Sigma_{\rho, h}$ converges to $\Sigma_{\rho, h_0}$ in $\C^{3,\alpha}$ as $h \to h_0$ precisely when $h_0$ is \emph{not} a jump height. It follows from the preceding lemma and Lemma \ref{max-prin} that $h_0$ is \emph{not} a jump height if and only if the map $h \to \mathcal{L}^n(\Omega_{\rho, h})$ is continuous at $h_0$, where $\mathcal{L}^n(\Omega_{\rho, h})$ is the volume of the enclosed region $\Omega_{\rho,h}$ defined in the proof of Lemma \ref{MOTS-existence}.  Note that this implies that there are at most countably many jump times $h_0 \in [- \Lambda, \Lambda]$.  

We also observe that the MOTS constructed in Lemma \ref{MOTS-existence} may be used to construct complete MOTS with good asymptotics.

\begin{lem}\label{asymptotic-decay}
For any choice of $\rho_j\to\infty$ and $h_j\in[-\Lambda,\Lambda]$, there exists a subsequence of $\Sigma_{\rho_j, h_j}$ that converges in $\C^{3,\alpha}$ on compact subsets of $M$ to a complete
$\C^{3,\alpha}$ properly embedded MOTS $\Sigma_\infty$.  Moreover, there exists a
constant $c\in[-\Lambda,\Lambda]$ such that outside a large compact subset
of $M$, $\Sigma_\infty$ can be written as the Euclidean graph
$\{x^n=f(x')\}$ of some $\C^{3,\alpha}$ function $f(x')=c+O^{3, \alpha}(|x'|^{3-n})$ in the
$(x^1, \ldots, x^{n-1}, x^n) = (x', x^n)$ coordinate system.

\begin{proof}
Existence of a subsequential limit $\Sigma_\infty$ and $\C^{3,\alpha}$ convergence
follow as in the proof of Lemma \ref{MOTS-convergence}. Note that the limit
$\Sigma_\infty$ is again $\lambda$-minimizing. Since each
$\Sigma_{\rho,h_\rho}$ lies between the horizontal planes $x^n=\pm\Lambda$,
so does $\Sigma_\infty$. The estimates below take part in the complement of
a large ball $B$ in $M$ where we have harmonic asymptotics for the metric
so that $g_{ij} = u^{\frac{4}{n-2}} \delta_{ij}$ where $u = 1 + O^{2, 
\alpha}(|x|^{2-n})$. It is not difficult to see from the corresponding
property of $\Sigma_{\rho_j}$ that the vertical projection of
$\Sigma_\infty$ onto the plane $\{x_n = 0\} \cap (M \setminus B)$ is
surjective. Note that (\ref{mean-curvature}) implies that the Euclidean
mean curvature of $\Sigma_\infty$ is $O(|x|^{1-n})$. The
$\lambda$-minimizing property of $\Sigma_\infty$ gives rise to an explicit
estimate of the form $O(|x'|^{-1})$ for the Euclidean area excess of
$\Sigma_\infty$ in large Euclidean balls centered at points $(x', x^n) \in
\Sigma_\infty$ and of radius $|x'|/2$. Together with the Allard regularity
theorem (the version in \cite[Theorem 24.2]{GMT} is particularly convenient
here), this estimate implies that outside some large compact set,
$\Sigma_\infty$ is the graph of a function $f(x')$ such that
$|f(x')|\le\Lambda$ and $f(x') = O^{1, \gamma}(1)$ for some $\gamma \in
(0,1)$.

Since $\Sigma_\infty$ is a MOTS we have that $H_{\Sigma_{\infty}} = -
\tr_{\Sigma_\infty}(k)$. As in \cite[p.\ 32]{Schoen:1989}, this translates
into a prescribed (Euclidean) mean curvature equation for $f$ via a
conformal change (\ref{mean-curvature}). The initial estimate $f =
O^{1, \gamma}(1)$, together with a computation of $\tr_{\Sigma_\infty} k$ as in the proof of Lemma \ref{barriers}, shows that the Euclidean mean curvature of $f$ is
$O^{0, \gamma} (|x'|^{1-n -\gamma})$ for some $\gamma \in (0, 1)$. Standard
asymptotic analysis as in \cite{Meyers:1963, Schoen:1983} shows that there
exists a constant $c \in [-\Lambda, \Lambda]$ such that $f(x') = c + O^{2, 
\gamma}(|x'|^{3 -n})$. (Note that there is no logarithm term when $n-1=2$ because
$f$ is bounded.) Repeating the above analysis with this information shows
that $f(x') = c + O^{3, \alpha}(|x'|^{3-n})$, as asserted. \end{proof}
\end{lem}

\begin{cor} When $n>3$, the hypersurface $\Sigma_{\infty}^{n-1} \subset M$
in Lemma \ref{asymptotic-decay} is asymptotically flat and has zero energy
with respect to the induced metric $g_\infty$. 
\end{cor}

The following lemma is a simple consequence of Proposition \ref{symmetrized-stability} and the fact that $\Sigma_\infty$ is a limit of stable MOTS. 
\begin{lem}\label{compact-stability} 
Let $\Sigma_\infty$ be a complete MOTS whose existence is established by Lemma \ref{asymptotic-decay}.
For any $v\in W_{\frac{3-n}{2}}^{1,2}(\Sigma_\infty)$, we have
\begin{equation*}
\int_{\Sigma_\infty} \left(|\nabla v |^2  + Q_{\Sigma_\infty} v^2\right)\,d\area
\geq 0.
\end{equation*}
\end{lem}
We omit the proof because it is strictly simpler than that of Lemma \ref{stability-infty} in Section \ref{section-complete}.  Specifically,
the proof of Lemma \ref{stability-infty} becomes a proof of Lemma \ref{compact-stability} by simply replacing $Z$  by an arbitrary compactly supported vector field on $M$ and replacing the use of Lemma \ref{stability} by Proposition \ref{symmetrized-stability}.


\section{The case $n=3$}\label{dimension-three}

We consider the base case $n=3$ of our inductive proof.   After the preparation of the previous section, the rest of the proof of the $n=3$ case is essentially the same as for the time-symmetric case in \cite{Schoen-Yau:1979-pmt1}, where minimal surfaces are replaced by MOTS.  

By Lemma \ref{asymptotic-decay}, we can extract a subsequential limit $\Sigma_\infty$ of $\Sigma_{\rho, 0}$ as $\rho\to\infty$.  Let $\Sigma'_\infty$ be the noncompact component of $\Sigma_\infty$.  Lemma \ref{compact-stability} implies that
\begin{align} \label{eq:stability-complete_slice}
	\int_{\Sigma'_{\infty}} \Big( | \nabla v |^2 +Q_{\Sigma'_{\infty}} v^2 \Big) \, d\mathcal{H}^2 \ge 0\end{align}
for every $v \in  W_{\frac{3-n}{2}}^{1,2}(\Sigma_\infty)$.  Noting that $\Sigma_\infty$ has quadratic area growth, we can use the logarithmic cut-off trick exactly as in \cite[page 54]{Schoen-Yau:1979-pmt1} to approximate the constant function $1$ on $\Sigma'_\infty$ by compactly supported functions in order to conclude that
\begin{equation}\nonumber
\int_{\Sigma'_\infty}Q_{\Sigma'_{\infty}}\, d \mathcal{H}^2 \ge 0. 
\end{equation}
The strict dominant energy condition then implies that 
\begin{equation}  \label{eqn:positive_Gaussian} \int_{\Sigma'_\infty} K_{\Sigma'_\infty}\, d \mathcal{H}^2 > 0,\end{equation}
where  $K_{\Sigma'_\infty}$ denotes the Gauss curvature.
On the other hand, just as in \cite{Schoen:1989}, the estimate $(g_\infty)_{ij}(x')-\delta_{ij}=O^2(|x'|^{-1})$ implies that the geodesic curvature of $\partial (\Sigma'_\infty\cap C_r)$ is $\kappa=\frac{1}{r}+O(r^{-2})$ while the length of $\partial (\Sigma'_\infty\cap C_r)$ is $2\pi r + O(1)$.  The Gauss-Bonnet Theorem tells us that
\[\int_{\Sigma'_\infty\cap C_r} K_{\Sigma'_{\infty}}\, d \mathcal{H}^2 = 2\pi\chi(\Sigma'_\infty\cap C_r) - \int_{\partial (\Sigma'_\infty\cap C_r) } \kappa \, d \mathcal{H}^1.\]
Combined with (\ref{eqn:positive_Gaussian}) and the asymptotics of  $\partial (\Sigma'_\infty\cap C_r)$ described above, for large~$r$, we obtain
\[0< 2\pi\chi(\Sigma'_\infty\cap C_r) - 2\pi.\]
Since $\Sigma'_\infty\smallsetminus C_r$ is a graph for large $r$, we  
know that $\Sigma'_\infty\cap C_r$ is connected, yielding a  
contradiction.
 \qed


\section{The case $3<n<8$}\label{higher-dimensions}

Let $3<n<8$.  We suppose that Theorem \ref{thm:PMT} holds in $n-1$ dimensions and that it fails for an $n$-dimensional initial data set $(M, g, k)$ as in Section \ref{section-construction}. For the reasons described in the introduction, the argument here is substantially different from the proof in the time-symmetric case.  

\subsection{The functional $\F$}\label{section-functional}

In this section we introduce a functional $\F$ that will be essential for our proof.
In order to motivate the definition of $\F$, consider the time-symmetric case when $k=0$ so that the MOTS $\{ \Sigma_{\rho, h}\}_{|h| \leq \Lambda}$ constructed in Section \ref{sec:existenceMOTS} are minimal hypersurfaces. An important step in the proof of the Riemannian positive mass theorem when $3 < n < 8$ \cite{Schoen:1989} is to pick $h_\rho$ such that  $\Sigma_{\rho, h_{\rho}}$ has least area in this family. 
Suppose for a moment that the family $\{\Sigma_{\rho,h}\}_{|h| \leq |\Lambda|}$ is actually a $\C^{3,\alpha}$ foliation of minimal hypersurfaces with a first-order deformation vector field $X = \varphi \nu + \hat X$ that is equal to  $\partial_n$ at $\partial C_\rho$.  Then 
\begin{align*}
\frac{d}{dh}\area(\Sigma_{\rho ,h})
&=  \int_{\Sigma_{\rho,h}}  (\Div_{\Sigma_{\rho,h}} X)\, d\area\\
&= \int_{\Sigma_{\rho,h}} \Div_{\Sigma_{\rho,h}} (\varphi \nu + \hat{X})\, d\area\\
&= \int_{\Sigma_{\rho,h}}   (H\varphi +  \Div_{\Sigma_{\rho,h}}\hat{X})\,d\area\\
&= \int_{\Sigma_{\rho,h}}   (\Div_{\Sigma_{\rho,h}}\hat{X})\,d\area\\
&=\int_{\partial\Sigma_{\rho,h}} \langle \hat{X},\eta\rangle\,d\barea\\
&=\int_{\partial\Sigma_{\rho,h}} \langle\partial_n,\eta\rangle\,d\barea.
\end{align*}
If $h_\rho \in (- \Lambda, \Lambda)$ minimizes areas as described above, then the \emph{first derivative} in $h$ of the integral above is nonnegative. This, along with the stability of $\Sigma_{\rho, h_\rho}$ among deformations that keep the boundary fixed, is all that is needed to finish the proof in the time-symmetric case \cite{Schoen:1989}. 

We now return to the general case. Instead of using the area functional, which is not adapted for application to MOTS, we will build our proof around the functional described below.
\begin{defin} Let $\Sigma$ be a compact hypersurface in $M$ whose boundary lies on some coordinate cylinder $\partial C_r$.  We let 
\begin{equation}\label{define-F}
\F(\Sigma)=\int_{\partial\Sigma} \langle \partial_n,\eta \rangle\,d\barea
\end{equation}
where $\eta$ is the outward unit normal of $\partial\Sigma$ in $\Sigma$.  Note that, using harmonic asymptotics,  one can easily see that 
\begin{equation}\label{alt-F}
\F(\Sigma)= \int_{\partial \Sigma} u^{\frac{2(n-1)}{n-2}} \eta_0^n \,d\barea_0
\end{equation}
where $\eta_0^n$ is the $n$-th component of the unit normal 
$\eta_0$ computed using the Euclidean metric, and $\mathcal{H}_0^{n-1}$ denotes Euclidean Hausdorff measure.
\end{defin}

The barrier planes $\{ x_n = \pm \Lambda\}$ give us a sign on $\F(\Sigma_{\rho,\pm\Lambda})$:

\begin{lem}\label{F-sign} For any $\rho$ sufficiently large, 
\[ \F(\Sigma_{\rho,-\Lambda})< 0< \F(\Sigma_{\rho,\Lambda}).\]
\end{lem}
\begin{proof}
From Lemma \ref{MOTS-existence} we know that $\Sigma_{\rho,\Lambda}$ lies below the plane $\{x^n=\Lambda\}$ in $C_\rho$. The strong maximum principle Lemma \ref{max-prin} implies that they cannot meet tangentially at their common boundary $\Gamma_{\rho, \Lambda}$. Hence the Euclidean outward unit normal of $\partial\Sigma_{\rho,\Lambda}$ in $\Sigma_{\rho,\Lambda}$ satisfies $\eta_0^n>0$.  The inequality $\F(\Sigma_{\rho, \Lambda}) >0$ then follows from (\ref{alt-F}). The proof that $\mathcal{F}(\Sigma_{\rho, - \Lambda}) < 0$ is analogous. 
\end{proof}

Recall the definition of jump heights from Section \ref{sec:convergenceMOTS}.
\begin{lem} \label{lem:downwardjump} 
The function $h \mapsto \mathcal{F}(\Sigma_{\rho, h})$ is continuous at every $h_0 \in [-\Lambda, \Lambda]$ that is \emph{not} a jump height. If $h_0 \in
[-\Lambda, \Lambda]$ \emph{is} a jump height, then 
\[\lim_{h\nearrow h_0}\F(\Sigma_{\rho, h}) \ge\F(\Sigma_{\rho, h_0}) \ge  \lim_{h \searrow h_0}\F(\Sigma_{\rho, h}),\]
where both limits exist, and at least one of the inequalities above is strict.  In other words, there must be a downward jump discontinuity at every jump height.
\end{lem}
 \begin{proof}
Let $h_0\in[-\Lambda,\Lambda]$.  Then by Lemma \ref{MOTS-convergence}, $\lim_{h\nearrow h_0}\F(\Sigma_{\rho, h}) = \F(\underline{\Sigma}_{\rho, h_0})$ and  $\lim_{h \searrow h_0}\F(\Sigma_{\rho, h})= \F(\overline{\Sigma}_{\rho, h_0})$.  By definition, if $h_0$ is not a jump height, then both of these limits must equal $\F({\Sigma}_{\rho, h_0})$

Let $h_0$ be a jump height.  Since the family  $\{\Sigma_{\rho,h}\}_{|h|\leq\Lambda}$ is ordered, it is clear that $\underline{\Sigma}_{\rho, h_0} $ lies beneath $\Sigma_{\rho,h_0}$, which lies beneath $\overline{\Sigma}_{\rho, h_0}$.  Since they all share the common boundary $\Gamma_{\rho, h_0}$,  we have $(\underline {\eta})_0^n\ge \eta_0^n \ge (\overline{\eta})_0^n$, where $\underline{\eta}, \eta, \overline{\eta}$ are the outward normals of $\Gamma_{\rho, h_0}$ in $\underline{\Sigma}_{\rho, h_0}$, ${\Sigma}_{\rho, h_0}$, and $\overline{\Sigma}_{\rho, h_0}$, respectively.  Since $h_0$ is jump height, $\underline{\Sigma}_{\rho, h_0} \neq \overline{\Sigma}_{\rho, h_0}$, so the strong maximum principle (Lemma \ref{max-prin}) implies that at least one of the above inequalities is strict.  By the definition of $\F$ in (\ref{alt-F}),
\[ \F(\underline{\Sigma}_{\rho, h_0}) \ge \F(\Sigma_{\rho, h_0})\ge \F(\overline{\Sigma}_{\rho, h_0}),\]
where at least one of the inequalities is strict.
\end{proof}

We now compute the first variation of $\F$. In view of Lemmas \ref{F-sign} and \ref{lem:downwardjump} we may hope to find $h_\rho \in (-\Lambda, \Lambda)$ such that the derivative of $h \to \F(\Sigma_{\rho, h})$ at $h_\rho$ (defined in a suitably weak sense) is nonnegative.

\begin{prop}\label{prop-variation-F} 
Let $\Sigma$ be a compact hypersurface with unit normal $\nu$ in $M$ whose boundary lies on some $\partial C_r$.  Let $X$ be a $\C^1$ vector field along $\Sigma$ that is tangent to $\partial C_r$ along $\partial\Sigma$.  Let $Z$ be a vector field of $M$ such that $Z=\partial_n$ along $\partial C_r$. Then
\begin{equation}\label{variation-F}
D\F|_{\Sigma}(X)=\int_{\partial\Sigma} \langle \phi\nabla \varphi + G(X),\eta\rangle \,d\barea
\end{equation}
where
\begin{equation}\label{def-G}
G(X)= D_X Z - D_{\hat{Z}} \hat{X} +(\varphi H + \Div_{\Sigma} \hat{X}) \hat{Z}-\phi S (\hat{X})-\varphi S(\hat{Z})
\end{equation}
and where $X = \varphi \nu + \hat X$ and $Z = \phi \nu + \hat Z$ are the decompositions of $X$ and $Z$ into normal and tangential parts along $\Sigma$.
\end{prop}

\begin{proof} 
Note that $\F(\Sigma)=\int_{\partial\Sigma} \langle Z,\eta \rangle\,d\barea$.  Let $e_1,\ldots,e_{n-2}$ be a local orthonormal frame for the tangent space of $\partial\Sigma$.  We can differentiate $Z$, the outward unit normal $\eta$ of $\partial \Sigma$ in $\Sigma$, and the induced measure on $\partial\Sigma$ to obtain
\begin{multline}\label{three-terms}
D\F|_{\Sigma}(X)=\\
\int_{\partial\Sigma} \left[ \langle D_X Z, \eta\rangle + \left\langle Z, \langle D_\eta X,\nu\rangle\nu-\sum_{i=1}^{n-2}\langle D_{e_i} X, \eta\rangle e_i \right\rangle +\langle Z,\eta\rangle \Div_{\partial\Sigma} X    \right]\,d\barea.
\end{multline}
(The derivative of $\eta$ is computed by differentiating the orthogonality relations.)  
Along $\partial\Sigma$, we introduce the decomposition $\hat{Z}=\psi\eta+Z^{\partial}$ into components that are normal and tangential to $\partial\Sigma$.  The second term in the integrand of (\ref{three-terms}) is
\begin{align}
 \left\langle Z, \langle D_\eta X,\nu\rangle\nu-\sum_{i=1}^{n-2}\langle D_{e_i} X, \eta\rangle e_i \right\rangle 
 &=
 \langle Z,\nu \rangle\langle D_\eta X,\nu\rangle -\sum_{i=1}^{n-2}\langle Z, e_i\rangle\langle D_{e_i} X, \eta\rangle\notag\\ 
 &=  \phi  \langle D_\eta (\varphi\nu+\hat{X}),\nu\rangle
-\langle D_{Z^\partial} X,\eta\rangle\notag \\
 &=  \phi  (\nabla_\eta \varphi+\langle D_\eta \hat{X},\nu\rangle)
-\langle D_{Z^\partial} X,\eta\rangle\notag \\
 &=  \langle \phi  \nabla \varphi,\eta\rangle -\langle \phi  S(\hat{X}),\eta\rangle
-\langle D_{Z^\partial} X,\eta\rangle.  \label{second-term}
\end{align}
The third term in the integrand of (\ref{three-terms}) is
\begin{align}
\langle Z,\eta\rangle \Div_{\partial\Sigma} X 
&=\langle Z,\eta\rangle  (\Div_{\Sigma} X-\langle D_\eta X, \eta\rangle)\notag \\
&=\langle \hat{Z},\eta\rangle (\varphi H+ \Div_{\Sigma} \hat{X})  -\psi \langle D_\eta X, \eta\rangle\notag\\
&= \langle (\varphi H+ \Div_{\Sigma} \hat{X})\hat{Z},\eta\rangle-\langle D_{\psi\eta} X,\eta\rangle.
 \label{third-term}
\end{align}
Notice that the first term in the integrand of (\ref{three-terms}), the first two terms of (\ref{second-term}) and the first  term of (\ref{third-term}) combine to give
\begin{equation}\label{first-chunk}
\langle \phi  \nabla \varphi + D_X Z + (\varphi H + \Div_{\Sigma}\hat{X}) \hat{Z} -\phi S(\hat{X}),\eta\rangle.
\end{equation}
The remaining two terms, which are the last term of (\ref{second-term}) and the last term of (\ref{third-term}), combine to give
\begin{align}
-\langle D_{Z^\partial} X,\eta\rangle-\langle D_{\psi\eta} X,\eta\rangle
&= -\langle D_{\hat{Z}} X,\eta\rangle \notag\\
&= -\langle D_{\hat{Z}} (\varphi\nu+\hat{X}),\eta\rangle \notag\\
&= -\langle \varphi D_{\hat{Z}} \nu+ D_{\hat{Z}} \hat{X},\eta\rangle \notag\\
&= -\langle \varphi S(\hat{Z})+ D_{\hat{Z}} \hat{X},\eta\rangle.\label{second-chunk}
\end{align}
The result follows from combining (\ref{first-chunk}) and (\ref{second-chunk}).
\end{proof}


\subsection{Height picking and stability}\label{section-height-picking} 

The following lemma, whose proof we defer to Section \ref{not-smooth}, will stand in for the geometric inequality (\ref{area-inequality}) that was available in the time-symmetric case. 
\begin{lem}\label{choose-h-rho}
Let $\Sigma'_{\rho, h}$ denote the component of $\Sigma_{\rho, h}$ that contains the boundary $\Gamma_{\rho,h}$.
For every large $\rho$ there exists $h_\rho\in (-\Lambda,\Lambda)$ and a $\C^2$ vector field $X$ along $\Sigma'_{\rho,h_\rho}$ that is equal to $\partial_n$ along $\partial\Sigma'_{\rho,h_\rho}=\Gamma_{\rho,h_\rho}$   such that $\varphi=\langle X, \nu \rangle > 0$,
\begin{equation}\label{MOTS-variation}
 D\theta^+|_{\Sigma'_{\rho,h_\rho}} (X) = 0,
 \end{equation} 
and
\begin{equation}\label{h-rho-inequality}
D\F|_{\Sigma'_{\rho,h_\rho}} (X) \ge 0.
\end{equation}
\end{lem}
The proof of this lemma would be straightforward if the path $h\mapsto \Sigma_{\rho,h}$ of $\C^{3,\alpha}$ hypersurfaces were differentiable in $h$ (and if the $\Sigma_{\rho, h}$ were connected). By Lemma \ref{F-sign}, we could find $h_\rho$ such that 
\[\left.\frac{d}{dh} \F(\Sigma_{\rho,h})\right|_{h=h_\rho} \ge0.\]
We would then choose $X$ to be the first-order deformation field of the family $\Sigma_{\rho,h}$ at $h=h_\rho$. The preceding inequality  would turn into (\ref{h-rho-inequality}), and the fact that each $\Sigma_{\rho,h}$ is a MOTS would lead to (\ref{MOTS-variation}).  Unfortunately, $\Sigma_{\rho,h}$ need not be differentiable in $h$. In general, the family $\Sigma_{\rho,h}$ must contain jumps for topological reasons. From Lemma \ref{lem:downwardjump} we know that $\F({\Sigma}_{\rho, h})$ can only jump down at a jump height, so the presence of jumps does not cause problems for finding $h_\rho$ as described above.  However, even in the absence of jumps, the lack of differentiability in $h$ presents a technical challenge.

\begin{notation} For the remainder of this section, we will abbreviate $\Sigma'_{\rho,h_\rho}$ by $\Sigma_\rho$. \end{notation}  
Lemma \ref{choose-h-rho} allows us to conclude the following stability-like property, which the reader should compare to Proposition \ref{symmetrized-stability}.
\begin{lem}\label{stability}
Let $\rho$ be sufficiently large. Let $X$ and $\varphi$ be as in the statement of Lemma \ref{choose-h-rho}.  For every $\C^1$ function $v$ on $\Sigma_\rho$ that is equal to $\phi=\langle \partial_n,\nu\rangle$ along $\partial \Sigma_\rho$ we have that
\begin{equation} \label{translate-stability}
\int_{\Sigma_\rho} (|\nabla v|^2  + Qv^2 )\,d\area
+\int_{\partial\Sigma_\rho} \langle \bar{G}(X),\eta\rangle \,d\barea \geq 0
\end{equation}
where   
\begin{equation}\label{define-Gbar}
\bar{G}(X) = G(X)+\phi \varphi W.
\end{equation}
\end{lem}

\begin{proof}
We begin by following the argument in \cite{Galloway-Schoen:2006}.  Using equation (\ref{dtheta}) and the positivity of $\varphi$, we compute that
\begin{align*}
D\theta^+|_{\Sigma_{\rho}} (X)
&=-\Delta \varphi +2 \langle W, \nabla\varphi\rangle  + (\Div W - |W|^2+Q)\varphi\\
&=-\Delta\varphi  + |\nabla\log \varphi|^2\varphi -|W-\nabla\log\varphi|^2\varphi + (\Div W +Q)\varphi\\
&=-(\Delta \log \varphi)\varphi - |W-\nabla\log\varphi|^2\varphi + (\Div W +Q)\varphi \\
&= [\Div ( W-\nabla \log \varphi)] \varphi  -|W-\nabla\log\varphi|^2 \varphi + Q\varphi.
\end{align*}
Let $v\in \C^{1} (\Sigma)$ be equal to $\phi$ along $\partial\Sigma_\rho$.  We multiply the above equation by ${v^2}{\varphi}^{-1}$ to obtain
\begin{align}
{v^2}{\varphi}^{-1}D\theta^+|_{\Sigma_{\rho}} (X)
&=[\Div ( W-\nabla \log \varphi)]v^2 - |W-\nabla\log\varphi|^2 v^2 + Q v^2  \notag\\
&=\Div (v^2( W-\nabla \log \varphi))-  \langle  W-\nabla \log \varphi, 2v\nabla v\rangle \notag\\
&\quad  - |W-\nabla\log\varphi|^2 v^2+ Q v^2\notag \\
&= \Div(v^2 ( W-\nabla \log \varphi))\notag\\
&\quad + |(W-\nabla \log \varphi)v|^2+ |\nabla v|^2
-|(W-\nabla\log\varphi)v+\nabla v|^2\notag\\
&\quad - |W-\nabla\log\varphi|^2 v^2+ Qv^2 \notag\\
\begin{split} \label{throw-away} 
&=  \Div(v^2 ( W-\nabla \log \varphi))+ |\nabla v|^2+ Qv^2\\
&\quad -|(W-\nabla\log\varphi)v+\nabla v|^2.
\end{split}
\end{align}
Together with equation (\ref{MOTS-variation}), this implies that
\[ 0 \le  |\nabla v|^2+ Qv^2 +  \Div(v^2 ( W-\nabla \log \varphi)).\]
Using that $v=\phi=\varphi$ along $\partial\Sigma_\rho$, we estimate
\begin{align*}
0&\le \int_{\Sigma_\rho} (|\nabla v|^2  + Qv^2)\,d\area +\int_{\partial\Sigma_\rho}\langle v^2 (W-\nabla \log \varphi),\eta\rangle   \,d\barea \\
&= \int_{\Sigma_\rho} (|\nabla v|^2  + Qv^2)\,d\area +\int_{\partial\Sigma_\rho}\langle \phi\varphi  W - \phi\nabla \varphi,\eta\rangle   \,d\barea \\
&= \int_{\Sigma_\rho} (|\nabla v|^2  + Qv^2)\,d\area  +\int_{\partial\Sigma_\rho}\langle \bar{G}(X),\eta\rangle   \,d\barea 
-D\F|_{\Sigma_\rho}(X)\\
&\le \int_{\Sigma_\rho} (|\nabla v|^2  + Qv^2)\,d\area  +\int_{\partial\Sigma_\rho}\langle \bar{G}(X),\eta\rangle   \,d\barea
\end{align*}
where we used (\ref{variation-F}) and (\ref{h-rho-inequality}) in the third and fourth lines, respectively.
\end{proof}

\subsection{Analysis of the complete MOTS limit} \label{section-complete}

For the remainder of the proof of Theorem \ref{thm:PMT} we follow the strategy for the time-symmetric case treated in \cite{Schoen:1989} closely.
By Lemma \ref{asymptotic-decay}, there is a sequence $\rho_j\to\infty$ such that the $\Sigma_{\rho_j}$ converge locally in $\C^{3,\alpha}$ to a complete MOTS $\Sigma_\infty$ that has the properties described in Lemma \ref{asymptotic-decay}.

Lemma \ref{compact-stability} is strong enough to show that $\Sigma_\infty$ is conformal to a scalar-flat asymptotically flat manifold but it does not provide sufficient control on the change of mass effected by this conformal change. For that, we need to allow test functions that are asymptotic to $1$. To justify the use of such test functions we need some uniform control on the $\Sigma_\rho$'s as in the following lemma.

\begin{lem} \label{uniform-control} Let $Z$ be a $\C^2$ vector field on $M$ that is equal to $\partial_n$ outside a compact set and let $Z = \phi \nu + \hat Z$ be the decomposition of $Z$ into normal and tangential part along $\Sigma_\rho$. 
The following estimate holds uniformly in $\rho$ large.
\begin{align} \left( \int_{\Sigma_\rho\smallsetminus C_r} \left(|\nabla
\phi|^2  + Q_{\Sigma_\rho}\phi^2\right)\,d\area 
+\int_{\partial(\Sigma_\rho\smallsetminus C_r)}
\langle\bar{G}(Z),\eta\rangle\,d\barea\right)
= O(r^{-1}).
\end{align}
 \end{lem}
\begin{proof} The estimates 
\begin{align}
 D\theta^+|_{\Sigma_\rho}(Z) &=O(|x|^{-n})\label{theta-estimate}  \\ 
D\F|_{\Sigma_\rho\cap C_r}(Z)&=O(r^{-1})\label{F-estimate} 
\end{align}
hold uniformly in $\rho$ large. To see this, we use the harmonic asymptotics and formula (\ref{mean-curvature}) to obtain that
\[H = u^{\frac{-2}{n-2}}\left(H_0
+\tfrac{2(n-1)}{n-2}u^{-1}\nabla_{\nu_0}u\right)\]
where $H_0$ and $\nu_0$ are the mean curvature and upward unit normal with
respect to the Euclidean metric.  Clearly, $H_0$ and $\nu_0$ do not change
under vertical translation, and $H_0$ is bounded. Therefore, the decay of
$u$ implies that
$DH|_{\Sigma_\rho}(Z) =O(|x|^{-n})$.  The $\tr_{\Sigma_\rho} k $ term is
easier to handle, giving us \eqref{theta-estimate}.  To derive
\eqref{F-estimate}, we simply use \eqref{alt-F}, the fact that $\nu_0$ does
not change under vertical translation, and the decay of $u$.

We vary $\Sigma_\rho$ in the direction $Z$, which is
just vertical translation outside a compact set.  We  repeat the argument from the beginning of the proof of Lemma \ref{stability}, except that we use $\Sigma_\rho\smallsetminus C_r$ instead of $\Sigma_\rho$, the vector
field $Z$ instead of $X$, the function $\phi$ instead of $\varphi$, and that we
take $v=\phi$. On $\Sigma_\rho
\smallsetminus C_r$, equation (\ref{throw-away}) becomes
\begin{align*}
{\phi} D\theta^+|_{\Sigma_\rho} (Z)
&=  \Div(\phi^2  W- \phi \nabla \phi)+ |\nabla \phi|^2+
Q\phi^2-|W|^2\phi^2.
\end{align*}
Therefore, 
\begin{align*}
& \int_{\Sigma_\rho\smallsetminus C_r} (|\nabla \phi|^2  + Q
\phi^2)\,d\area +\int_{\partial(\Sigma_\rho \smallsetminus C_r)}\langle
\bar{G}(Z),\eta\rangle   \,d\barea\\
&\quad =  \int_{\Sigma_\rho\smallsetminus C_r} (|\nabla \phi|^2  + Q
\phi^2)\,d\area +\int_{\partial(\Sigma_\rho \smallsetminus C_r)}\langle
\phi^2  W- \phi \nabla \phi,\eta\rangle  \,d\barea\\
&\quad\quad +D\F|_{\Sigma_\rho}(Z)-D\F|_{\Sigma_\rho\cap C_r}(Z)\\
&\quad= \int_{\Sigma_\rho\smallsetminus C_r}\left( {\phi}
D\theta^+|_{\Sigma_\rho} (Z) + |W|^2\phi^2\right)\,  d\area\\
&\quad=O(r^{-1})
\end{align*}
where the last line follows from the \eqref{theta-estimate} and
\eqref{F-estimate}, decay of $W$, and volume control coming from the almost minimizing property of $\Sigma_\rho$.  
\end{proof}

The following lemma is the analytic consequence of our careful
height-picking in Section \ref{section-height-picking}.

\begin{lem}\label{stability-infty} 
Let $v$ be a function on $\Sigma_\infty$ such that $v-1\in
W_{\frac{3-n}{2}}^{1,2}(\Sigma_\infty)$.  Then
\begin{equation}\label{stability-infty-inequality}
\int_{\Sigma_\infty} \left(|\nabla v |^2  + Q_{\Sigma_\infty}
v^2\right)\,d\area
\geq 0.
\end{equation}
\end{lem}

\begin{proof} Following the notation of Lemma \ref{asymptotic-decay}, we
use coordinates $x'$ on $\Sigma_\infty \smallsetminus B$ where $B$ is a large
compact subset of $M$. Lemma \ref{asymptotic-decay} gives
that 
\[Q_{\Sigma_\infty} =
\tfrac{1}{2}R_{\Sigma_\infty}-\mu-J(\nu_{\Sigma_\infty})-\tfrac{1}{2}|k_{\Sigma_\infty}+B_{\Sigma_\infty}|^2
= O(|x'|^{{-n}}).\]
Using that the volume of $\Sigma_\infty$ grows like that of 
$\mathbb{R}^{n-1}$, it follows that 
\begin{equation}\label{finite-integral}
\int_{\Sigma_\infty} \left(|\nabla v |^2  + |Q_{\Sigma_\infty}|
v^2\right)\,d\area< \infty
\end{equation}
provided that $v-1\in W_{\frac{3-n}{2}}^{1,2}(\Sigma_\infty)$.

Below we abbreviate $\Sigma_{\rho_j} = \Sigma'_{\rho_j,
h_{\rho_j}}$ by $\Sigma_j$. We also use a subscript $j$ on geometric quantities to indicate that they are computed with respect to $\Sigma_j$. We
use $\eta_j$ to denote the outward unit normal of $\partial(\Sigma_j \cap
C_r)$ in $\Sigma_j$ for large $r < \rho_j$ such that $\partial C_r$ is transverse to each
$\Sigma_j$. Fix a $\C^2$ vector field $Z$ on $M$ that agrees with
$\partial_n$ outside a compact set. Let $\phi_\infty = \langle
\nu_{\Sigma_\infty}, Z \rangle$ and $\phi_j = \langle \nu_{\Sigma_j}, Z
\rangle$. It follows from Lemma \ref{asymptotic-decay} that $\phi_\infty
- 1= O^{1, \alpha}(|x'|^{2-n})  \in W^{1, 2}_{(3-n)/2}(\Sigma_\infty)$ and that
$\bar{G}_\infty(Z)=O(|x'|^{1-n})$ where $\bar{G}_\infty$ is as in \eqref{define-Gbar}. We have that

\begin{align*}
&\int_{\Sigma_\infty} \left(|\nabla\phi_\infty|^2  + Q_\infty
\phi_\infty^2\right)\,d\area\\
&=\lim_{r\to\infty}\int_{\Sigma_\infty\cap C_r}
\left(|\nabla\phi_\infty|^2  + Q_\infty \phi_\infty^2\right)\,d\area\\
&=\lim_{r\to\infty}\left(
\int_{\Sigma_\infty\cap C_r} \left(|\nabla\phi_\infty|^2  + Q_\infty
\phi_\infty^2\right)\,d\area 
+\int_{\partial(\Sigma_\infty\cap C_r)} \langle
\bar{G}_\infty(Z),\eta_\infty\rangle  \,d\barea\right)\\
&=\lim_{r\to\infty}\lim_{j\to\infty}\left(
\int_{\Sigma_j\cap C_r} \left( |\nabla\phi_j|^2  + Q_j
\phi_j^2\right)\,d\area 
+\int_{\partial(\Sigma_j\cap C_r)} \langle \bar{G}_j(Z),\eta_j\rangle 
\,d\barea\right)\\
&=\lim_{j\to\infty}\left(
\int_{\Sigma_j} \left(|\nabla\phi_j|^2  + Q_j \phi_j^2\right)\,d\area 
+\int_{\partial\Sigma_j} \langle \bar{G}_j(Z),\eta_j\rangle 
\,d\barea\right)\\
&\ge 0,
\end{align*}
where the last inequality follows from Lemma \ref{stability}, and the
fourth equality follows from Lemma \ref{uniform-control}. It follows that
(\ref{stability-infty-inequality}) holds for $v = \phi_\infty$.  Moreover, since the argument works for
any $Z$ that equals $\partial_n$ outside a compact set, (\ref{stability-infty-inequality}) holds for
all test functions that agree with $\phi_\infty$ outside a compact set.

We now argue by density that (\ref{stability-infty-inequality})  holds for any function $v$ such that
$v - \phi_\infty \in W^{1, 2}_{(3-n)/2}(\Sigma_\infty)$. Note that
$\C_c^{3,\alpha} (\Sigma_\infty)$ is dense in $W^{1,
2}_{(3-n)/2}(\Sigma_\infty)$. Let $v_i - \phi_\infty$ be a sequence of functions in 
$\C^{3,\alpha}_c(\Sigma_\infty)$ that converges to $v-\phi_\infty$ in $W^{1, 2}_{(3-n)/2}(\Sigma_\infty)$. It is
straightforward to check that 
\begin{align*}
0 &\leq \liminf_{i \to \infty} \int_{\Sigma_\infty}\left( |\nabla v_i|^2  +
Q_{\infty} v_i^2\right)\, d\area \\  
&=\liminf_{i \to \infty} \int_{\Sigma_\infty}\left[ ( |\nabla v|^2  +
Q_{\infty} v^2) + (2\nabla v\cdot\nabla (v_i-v)+2Q_\infty v(v_i-v))\right.\\
&\quad\left.+
( |\nabla (v_i-v)|^2  +
Q_{\infty} (v_i-v)^2)\right]\, d\area \\  
&=  \int_{\Sigma_\infty}\left( |\nabla v|^2 +
Q_{\infty} v^2\right)\, d \area
\end{align*}
where the cross terms vanish because of \eqref{finite-integral}. This implies the result since $\phi_\infty
- 1\in W^{1, 2}_{(3-n)/2} (\Sigma_\infty)$.
\end{proof}


We are now ready to derive a contradiction to our induction hypothesis.  Without loss of generality, we throw away all compact components of $\Sigma_\infty$, noting that Lemmas \ref{compact-stability} and \ref{stability-infty} still hold on the remaining asymptotically flat component.  
The strict dominant energy condition implies that  $Q_\infty< \tfrac{1}{2}R_{\Sigma_\infty}$. Lemma  \ref{compact-stability} implies that for any nonzero $v\in W_{(3-n)/2}^{1,2}(\Sigma_\infty)$, 
 we have
\[ \int_{\Sigma_\infty} \left(|\nabla v |^2  + \tfrac{1}{2}R_{\Sigma_\infty} v^2\right)\,d\area
> 0.\]
Using that $n>3$, we have
\begin{equation}\label{conformal-stability}
 \int_{\Sigma_\infty} \left(|\nabla v |^2  +  \tfrac{n-3}{4(n-2)} R_{\Sigma_\infty} v^2\right)\,d\area
> 0
\end{equation}
for all nonzero $v\in W_{(3-n)/2}^{1,2}(\Sigma_\infty)$.
This implies that the conformal Laplacian
\[\Delta_{\Sigma_\infty} - \tfrac{n-3}{4(n-2)} R_{\Sigma_\infty} : W_{\frac{3-n}{2}}^{1,2}(\Sigma_\infty)\to W_{\frac{-1-n}{2}}^{-1,2}(\Sigma_\infty)\]
is an isomorphism.  In particular, we can find some nonzero $v\in W_{(3-n)/2}^{1,2}(\Sigma_\infty)$ such that
\[\Delta_{\Sigma_\infty} v - \tfrac{n-3}{4(n-2)} R_{\Sigma_\infty}v= \tfrac{n-3}{4(n-2)} R_{\Sigma_\infty}. \]
Setting $w=1+v$, we have that
\[\Delta_{\Sigma_\infty} w - \tfrac{n-3}{4(n-2)} R_{\Sigma_\infty}w=0.\]
Note that $w\in\C^{2,\alpha}_{loc}$ by elliptic regularity. By \cite[Theorem 2]{Meyers:1963}, we have that $w(x')=1+O^{2, \alpha}(|x'|^{3-n})$.  Let $g_\infty$ be the induced metric on $\Sigma_\infty$. Applying (\ref{scalar-curvature}) in dimension $n-1$, the conformal metric $w^{\frac{4}{n-3}}g_\infty$ on $\Sigma_\infty$ is  asymptotically flat with zero scalar curvature.  Using (\ref{mass-change-u}) in dimension $n-1$, this metric has energy 
\begin{align*}
E(w^{\frac{4}{n-3}}g_\infty) &= E(g_\infty)-  \tfrac{2}{(n-3)\omega_{n-2}}\lim_{r\to\infty}\int_{\partial(\Sigma_\infty\cap C_r)} w \nabla_{\nu} w \,d\area\\
&= 0- \tfrac{2}{(n-3)\omega_{n-2}} \int_{\Sigma_\infty}( |\nabla w|^2 +w\Delta w )\, d\area\\
&=  \tfrac{-2}{(n-3)\omega_{n-2}} \int_{\Sigma_\infty}\left( |\nabla w|^2 + \tfrac{n-3}{4(n-2)} R_{\Sigma_\infty}w^2\right)\, d\area.
\end{align*}
Since $w-1\in W_{(3-n)/2}^{1,2}(\Sigma_\infty)$ is nonzero, we can apply  Lemma \ref{stability-infty} and the short argument used to derive (\ref{conformal-stability}) to conclude that
\[ \int_{\Sigma_\infty} \left(|\nabla w |^2  +  \tfrac{n-3}{4(n-2)} R_{\Sigma_\infty} w^2\right)\,d\area
> 0.\]
It follows that $E(w^{\frac{4}{n-3}}g_\infty)<0$.  This contradicts the time-symmetric case of Theorem \ref{thm:PMT} in dimension $n-1$. In particular, it contradicts our induction hypothesis. 


\subsection{Proof of Lemma \ref{choose-h-rho}}\label{not-smooth}

We define
\begin{equation}\label{define-h-rho}
h_\rho=\inf \{ h\in[-\Lambda,\Lambda]  \,|\, \mathcal{F}(\Sigma_{\rho,h})>0\}.
\end{equation}
Our goal for this section is to prove that this choice of $h_\rho$ satisfies the conclusion of Lemma \ref{choose-h-rho}.  For now let us assume that $\Sigma_{\rho,h_\rho}$ is connected.

Using Lemmas \ref{F-sign} and  \ref{lem:downwardjump} it is easy to see that $h_\rho$ exists, lies in $(-\Lambda, \Lambda)$, is not a jump height, and that $\F(\Sigma_{\rho,h_\rho})=0$.   Although the path $h \mapsto \Sigma_{\rho,h}$ of $\C^{3,\alpha}$ hypersurfaces need not be differentiable in $h$ at $h_\rho$, the continuity of the map $h\mapsto \Sigma_{\rho,h}$ at $h=h_\rho$ allows us to use the inverse function theorem to describe the family $\Sigma_{\rho,h}$ precisely for $h$ near $h_\rho$, as we will see below.  We note that the work of B.\ White \cite{White:1987} on the moduli space of minimal submanifolds with boundary in $\rr^n$ is useful for understanding the family $\Sigma_{\rho,h}$.  However, since we do not need the full power of \cite{White:1987}, we choose to use a simpler approach similar to that of \cite{Galloway:2008} in our Case 2 below. 

Since $\Sigma_{\rho,h_\rho}$ is transverse to $\partial C_\rho$, we can find a $\C^{2,\alpha}$ vector field $\partial_\tau$ such that $\langle\partial_\tau,\nu\rangle >0$ along $\Sigma_{\rho,h_\rho}$, and $\partial_\tau= Z=\partial_n$ at $\Gamma_{\rho,h_\rho}$.  By integrating $\partial_\tau$, there exists a relatively open neighborhood $U$ of $\Sigma_{\rho,h_\rho}$ in $C_\rho$ and a $\C^{2,\alpha}$ diffeomorphism 
$F: \Sigma_{\rho,h_\rho}\times (-\delta,\delta)\to U$ such that $F(\cdot, 0)$ is the identity map and  such that $F_*(\frac{\partial}{\partial\tau})=\partial_\tau$.   In particular, $F(\cdot,\tau)$ maps $\partial\Sigma_{\rho,h_\rho}=\Gamma_{\rho,h_\rho}$ to $\Gamma(h_\rho+\tau)$ by vertical translation.   

Every $\C^{2,\alpha}$ function $w:\Sigma_{\rho,h_\rho} \to (-\delta,\delta)$ gives rise to a graph in $\Sigma_{\rho,h_\rho}\times (-\delta,\delta)$ whose image under  the diffeomorphism $F$ is a $\C^{2,\alpha}$ hypersurface in $U$ denoted by $\graph[w]$ whose boundary lies on $\partial C_\rho$.  Since $h_\rho$ is not a jump height, Lemma \ref{MOTS-convergence} tells us that for each $h$ sufficiently close to $h_\rho$ the MOTS $\Sigma_{\rho,h}$ coincides with $\graph[w]$ for a unique $\C^{2,\alpha}$ function $w$.   Moreover, the function $w$ converges to $0$ in $\C^{2,\alpha}$ as $h$ approaches $h_\rho$.  

The operator $L_{\Sigma_{\rho,h_\rho}}$ on $\Sigma_{\rho,h_\rho}$ defined by (\ref{define-L}) has a principal Dirichlet eigenvalue, which is nonnegative because  $\Sigma_{\rho,h_\rho}$  is stable.  We consider two cases.\\

\textbf{Case 1:}  The principal eigenvalue of $L_{\Sigma_{\rho,h_\rho}}$ is positive.  \\

\noindent Define the map\footnote{Technically, it is only defined on some neighborhood of the origin.} 
$\Psi: \C_0^{2,\alpha}(\Sigma_{\rho,h_\rho})\times \rr \to  \C^{0,\alpha}(\Sigma_{\rho,h_\rho})\times \rr$
by
\[\Psi (w, s) = (\theta^+_{\graph[w+s]}, s).\]
Using (\ref{dtheta}) and (\ref{define-L}), we see that
\[ D\Psi|_{(0,0)} (w', s') = (L_{\Sigma_{\rho,h_\rho}}\langle(w'+s') \partial_\tau, \nu\rangle, s').\]
By assumption, $L_{\Sigma_{\rho,h_\rho}}:\C_0^{2,\alpha}\to \C^{0,\alpha}$ is an isomorphism. It follows that $D\Psi|_{(0,0)} : \C_0^{2,\alpha}(\Sigma_{\rho,h_\rho})\times \rr \to  \C^{0,\alpha}(\Sigma_{\rho,h_\rho})\times \rr$ is an isomorphism as well.  The inverse function theorem gives a $\C^{1}$ function 
$\Phi: (-\epsilon,\epsilon)\to  \C_0^{2,\alpha}(\Sigma_{\rho,h_\rho})$ such that $\Phi(s)$ is the unique small solution of the equation $\Psi(\Phi(s), s)=(0,s)$. Its graph is the unique $\C^{2,\alpha}$ nearby MOTS with boundary $\Gamma_{\rho,h_{\rho}+s}$.  By continuity of $\Sigma_{\rho,h}$ at $h=h_\rho$ and  uniqueness, we have that $\Sigma_{\rho,h_\rho+s}=\graph[\Phi(s)+s]$ for small $s$.  It follows that $h\mapsto\Sigma_{\rho,h}$ is a $\C^{1}$ path into the space of $\C^{2,\alpha}$ hypersurfaces near $h = h_\rho$ to which the straightforward argument described in Section \ref{section-height-picking} applies.  Specifically, consider the $C^{2, \alpha}$ first-order deformation field $X= (\frac{d\Phi}{ds}+1)\partial_\tau$  of the family $\Sigma_{\rho,h}$.  Then $D\theta^+|_{\Sigma_{\rho,h_\rho}}(X)=\frac{d}{dh}\theta^+_{\Sigma_{\rho,h}}|_{h=h_\rho}=0$ because each $\Sigma_{\rho,h}$ is a MOTS.  That $D\F|_{\Sigma_{\rho,h_\rho}}(X)=\left.\frac{d}{dh}\F(\Sigma_{\rho,h})\right|_{h=h_\rho}\ge 0$ follows from the construction of $h_\rho$.  The normal component $\varphi$ of $X$ along $\Sigma_{\rho,h}$ is nonnegative and positive on $\Gamma_{\rho,h}$ by construction, and thus positive everywhere by the strong maximum principle applied to the linear equation $L_{\Sigma_{\rho,h_\rho}} \varphi=0$.  This completes the proof of Lemma \ref{choose-h-rho} for Case 1. \\
 
\textbf{Case 2:}  The principal eigenvalue of $L_{\Sigma_{\rho,h_\rho}}$ is zero. \\

\noindent In this case, since $\Sigma_{\rho,h_\rho}$ is connected, $L_{\Sigma_{\rho,h_\rho}}$  has a one-dimensional kernel in $\C_0^{2,\alpha}$ that is generated by a function that is positive away from the boundary.  The same is then true for the adjoint $L_{\Sigma_{\rho,h_\rho}}^*$.  
  We define the map
$\Psi: \C_0^{2,\alpha}(\Sigma_{\rho,h_\rho})\times \rr^2\to  \C^{0,\alpha}(\Sigma_{\rho,h_\rho})\times \rr^2$ by
\[\Psi (w, \kappa, s) = \left( - \kappa + \theta^+_{\graph[w+s]} ,  V(\graph[w+s]), s\right)\]
where $V(\Sigma)$ denotes the signed volume of the region of $C_\rho$ lying above $\Sigma_{\rho,h_\rho}$ and below $\Sigma$.  Then 
\[ D\Psi|_{(0,0,0)} (w', \kappa', s') = \left(L_{\Sigma_{\rho,h_\rho}}\langle (w'+s')\partial_\tau, \nu\rangle -\kappa' , \int_{\Sigma_{\rho,h_\rho}} \langle (w'+s')\partial_\tau, \nu\rangle\,d\area ,  s'\right).\]
We claim that $D\Psi|_{(0,0,0)}$ is injective.  If $D\Psi|_{(0,0,0)} (w', \kappa', s')=(0,0,0)$, then obviously $s'=0$, and since the image of $\C^{2,\alpha}_0$ under $L_{\Sigma_{\rho,h_\rho}}$ is orthogonal to the kernel of $L_{\Sigma_{\rho,h_\rho}}^*$, $\kappa'=0$.  Finally, $\langle w'\partial_\tau, \nu\rangle$ is in the kernel of $L_{\Sigma_{\rho,h_\rho}}$ and has zero integral, and is therefore zero, proving the claim.  Since $D\Psi|_{(0,0,0)}$ has index zero, it is also an isomorphism. By the inverse function theorem, there exists a $\C^{1}$ function
\[
(\Phi_1,\Phi_2): (\epsilon,\epsilon)\times  (\epsilon,\epsilon)\to \C_0^{2,\alpha}(\Sigma_{\rho,h_\rho})\times\rr
\]
such that $(\Phi_1(\xi,s),\Phi_2(\xi,s))$ is the unique small solution of the equation
\[\Psi(\Phi_1(\xi,s), \Phi_2(\xi,s), s)=(0,\xi,s).\]  In particular, $\graph[\Phi_1(\xi,s)+s]$ is the unique $\C^{2,\alpha}$-nearby constant $\theta^+$ hypersurface whose boundary is $\Gamma_{\rho,h_\rho+s}$ and whose signed volume is $\xi$.  Define $\xi(s):=V(\Sigma_{\rho,h_\rho+s})$.  By continuity of $\Sigma_{\rho,h}$ at $h=h_\rho$ and the uniqueness, it must be the case that $\Sigma_{\rho,h_\rho+s}=\graph[\Phi_1(\xi(s),s)+s]$ for small $s$.  Note that since $\Sigma_{\rho,h_\rho+s}$ lies strictly above $\Sigma_{\rho,h_\rho}$, the function $\Phi_1(\xi(s),s)+s$ must be positive.  The complication here is that we do not know that $\xi(s)$ depends on $s$ in any nice way.  Below, we will see that this does not matter.

By the construction of $h_\rho$, we can choose a sequence $s_k\searrow 0$ such that 
\[\F(\Sigma_{\rho,h_\rho+s_k})>0.\]  
Let us pass to a subsequence such that the unit vector in the direction of $(\xi(s_k) ,s_k)$ converges.  We consider two subcases. \\

\textbf{Case 2a:} The ratio $\frac{\xi(s_k)}{s_k}$ converges to a finite number as $k\to\infty$. \\

\noindent The basic idea here is while we cannot take derivatives as in Case 1, we can take subsequential limits of difference quotients instead.
The hypothesis of this subcase implies that 
$\frac{1}{s_k}[\Phi_1 (\xi(s_k), s_k)+s_k]$ converges to a $\C^{2,\alpha}$ function $\bar{\varphi}$.  Therefore we can take the limit of $\frac{1}{s_k}\theta^+_{\graph[\Phi_1 (\xi(s_k), s_k)+s_k]}=0$ as $k \to \infty$ to obtain 
\[L_{\Sigma_{\rho,h_\rho}} (\langle \bar{\varphi}\partial_\tau, \nu\rangle)=0.\]  We claim that the conclusion of Lemma \ref{choose-h-rho} holds for $X=\bar{\varphi}\partial_\tau\in\C^{2,\alpha}$.  The previous equation tells us that $D\theta^+|_{\Sigma_{\rho,h_\rho}}(X)=0$. Moreover, the normal component $\varphi=\bar{\varphi}\langle \partial_\tau, \nu\rangle$ is nonnegative and equal to $\langle Z, \nu\rangle>0$ at $\Gamma_{\rho,h_\rho}$. In fact, we see that $\varphi$ is positive by the strong maximum principle applied to the operator $L_{\Sigma_{\rho,h_\rho}}$.  Finally, since $\F(\Sigma_{\rho,h_\rho+s_k})>0$, we have that
\begin{align*}
D\F|_{\Sigma_{\rho,h_\rho}}(X)&=\lim_{k\to\infty}\frac{1}{s_k}[\F(\graph[\Phi_1 (\xi(s_k), s_k)+s_k])-\F(\Sigma_{\rho,h_\rho})]\\
&=\lim_{k\to\infty}\frac{1}{s_k}[\F(\Sigma_{\rho,h_\rho+s_k})-\F(\Sigma_{\rho,h_\rho})]\\
&\ge 0. 
\end{align*}

\textbf{Case 2b:}  The ratio $\frac{s_k}{\xi(s_k)}$ converges to zero as $k\to\infty$.\\

\noindent We will show that this is impossible.  The assumption implies that the quotients
$\frac{1}{\xi(s_k)}[\Phi_1 (\xi(s_k), s_k)+s_k]$ converge to a $\C^{2,\alpha}$ function $\bar{\varphi}$ that vanishes along $\Gamma_{\rho,h_\rho}$.  As in Case 2a, we conclude that 
\[L_{\Sigma_{\rho,h_\rho}} (\langle \bar{\varphi}\partial_\tau, \nu\rangle)=0.\]
Let $X=\bar{\varphi}\partial_\tau\in\C^{2,\alpha}$.  Its normal component $\varphi = \bar{\varphi}\langle \partial_\tau, \nu\rangle$ is nonnegative with zero boundary values. It follows from the strong maximum principle that either $\varphi$ is identically zero, or else $\nabla_\eta\varphi<0$ along $\Gamma_{\rho,h_\rho}$.  Assume the latter. Taking the limit of the equation $V(\Sigma_{\rho,h_\rho+s_k}) /\xi(s_k)=1$ we obtain $\int_{\Sigma_{\rho,h_\rho}} \varphi\,d\area=1$.  As in Case 2a, 
\begin{align*}
D\F|_{\Sigma_{\rho,h_\rho}}(X)&=\lim_{k\to\infty}\frac{1}{\xi(s_k)}[\F(\graph[\Phi_1 (\xi(s_k), s_k)+s_k])-\F(\Sigma_{\rho,h_\rho})]\\
&=\lim_{k\to\infty}\frac{1}{\xi(s_k)}[\F(\Sigma_{\rho,h_\rho+s_k})-\F(\Sigma_{\rho,h_\rho})]\\
&\ge 0. 
\end{align*}
On the other hand, only the second term in (\ref{three-terms}) contributes to $D\F|_{\Sigma_{\rho,h_\rho}}$ because $X$ vanishes along $\Gamma_{\rho,h_\rho}$. Thus  
\begin{align*}
D\F|_{\Sigma_{\rho,h_\rho}}(X)&=\int_{\Gamma_{\rho,h_\rho}} \langle Z, \langle D_\eta X, \nu\rangle \nu \rangle d \mathcal{H}^{n-2} \\
&=\int_{\Gamma_{\rho,h_\rho}} \phi  \nabla_\eta \varphi \,d\barea\\ &<0. 
\end{align*}
This contradiction shows that Case 2b cannot occur. \\

The proof of Lemma \ref{choose-h-rho} in the case where $\Sigma_{\rho, h_\rho}$ is connected is now finished. Assume now that $\Sigma_{\rho, h_\rho}$ is not connected. Let $\Sigma'_{\rho, h}$ denote the component of $\Sigma_{\rho, h}$ that contains the boundary.  Observe that $\Sigma'_{\rho,h}$ converges to $\Sigma'_{\rho,h_\rho}$ in $\C^{3,\alpha}$ as $h\to h_\rho$ because $h_\rho$ is not a jump height. Moreover, $\F(\Sigma'_{\rho,h})=\F(\Sigma_{\rho,h})$ for all $h\in[-\Lambda, \Lambda]$. From this it is easy to see that the above argument can be carried out  with $\Sigma'_{\rho,h_\rho}$ in place of $\Sigma_{\rho, h}$.


\section{The density theorem}\label{section-density}

Recall that $\pi = k -(\tr_g k)g$ and that 
\[ n \geq 3, \quad p>n, \quad q \in ((n-2)/2, n-2), \quad q_0 >0\quad \text{ and } \quad \alpha \in (0, 1 -n/p].
\]  
It will be convenient to express initial data in terms of $\pi$ rather than $k$.  Abusing terminology slightly, we will refer to $(M,g,\pi)$ as an initial data set in this section. We denote by $g_{\mathbb{E}}$ a fixed smooth symmetric $(0, 2)$-tensor that coincides with the Euclidean metric on $M\smallsetminus K \cong \mathbb{R}^n \smallsetminus B$ throughout this section.

\begin{thm} [Density theorem] \label{th:density-theorem-section}
Let $(M, g, \pi)$ be an $n$-dimensional asymptotically flat initial data set of type $(p,q,q_0,\alpha)$ such that the dominant energy condition $\mu \geq |J|_{g}$ holds. Let $\epsilon > 0$. There are asymptotically flat initial data $(\bar{g}, \bar{\pi})$ with harmonic asymptotics and of type $(p,q,q_0,\alpha)$ on $M$ such that
\[
	\| g -\bar{g} \|_{W^{2,p}_{-q} } < \epsilon, \quad \| \pi - \bar{ \pi } \|_{W^{1,p}_{-1-q}} < \epsilon, \quad | E-\bar{E}| < \epsilon, \quad |P-\bar{P}|< \epsilon,
\]
and such that the strict dominant energy condition 
\[
	\bar{\mu} > | \bar{J} |_{\bar{g}}
\]
holds. 
\end{thm}

\begin{remark}
If we assume appropriate higher regularity for $(M,g,\pi)$, then $(M, \bar{g},\bar{\pi})$ will have the same regularity.  This follows from applying Schauder estimates throughout the proof. Our argument shows that $(\bar g, \bar \pi)$ can be taken to be of type $(p, q, q_0', \alpha)$ for any given $q_0'>q_0$.   
\end{remark}

The proof of Theorem \ref{th:density-theorem-section} consists of two general constructions. In Section \ref{subse:dec} we deform $(g,\pi)$ to initial data $(\hat g, \hat \pi)$ such that $\hat{\mu} > (1+\gamma) | \hat{J}|_{\hat{g}}$ for some $\gamma >0$. 
In Section \ref{subse:harmonic-asymptotics} we apply a cut-off argument to perturb $(\hat g, \hat \pi)$ to harmonic asymptotics as in \cite[Theorem 1]{Corvino-Schoen:2006} and argue that we can preserve the strict dominant energy condition in the process. 

The ADM energy and linear momentum are continuous on the space of asymptotically flat initial data sets in the following sense. 

\begin{prop} [\emph{Cf.\ }{\cite[Proposition 2.4]{Huang:2011}}] \label{prop:cont} 
Let $(g, \pi)$ and $(\bar g, \bar \pi)$ be asymptotically flat initial data of type $(p,q,q_0,\alpha)$.  
Let $\epsilon>0$. There exists $\delta>0$ depending only on $\epsilon, n, p, q, q_0, \|(g - g_{\mathbb{E}}, \pi)\|_{W^{2,p}_{-q} \times W^{1,p}_{-q-1}}$, 
and $\|(\mu, J) - (\bar \mu, \bar J)\|_{{L^{1}_{-n-q_0/2}}}$  such that if 
\begin{align*}
	 \| g - \bar{g} \|_{W^{2,p}_{-q}} \le \delta \quad \mbox{and} \quad \| \pi - \bar{\pi} \|_{W^{1,p}_{-1-q}} \le \delta,
\end{align*}
then
\begin{align*}
	| E-\bar{E}| < \epsilon \quad \mbox{and}  \quad |P-\bar{P}|< \epsilon.
\end{align*}
 
\end{prop}
The proof of this fact is now standard and goes back to  \cite[p.\ 50]{Schoen-Yau:1981-asymptotics} (for $E$ only) and \cite[p.\ 198]{Corvino-Schoen:2006} in the case of vacuum data. We include the argument for the sake of completeness.  
\begin{proof}
By the definition of $E$ and the divergence theorem, we have that 
\begin{align*}
	2(n-1)\omega_{n-1} {E} &= \lim_{r \rightarrow \infty} \int_{|x|=r} \sum_{i,j = 1}^n ({g}_{ij,i} -{g}_{ii,j} )\nu_0^j  \\
	&=  \int_{|x|=r} \sum_{i,j = 1}^n ({g}_{ij,i} -{g}_{ii,j} )  \nu_0^j\\
	&\quad + \int_{|x| \ge r} \sum_{i,j = 1}^n ({g}_{ij,ij} - {g}_{ii,jj})\, 
\end{align*}
for all $r$ sufficiently large, and similarly for $\bar E$. All integrals here are with respect to the Euclidean metric. Note that $\sum_{i,j=1}^n({g}_{ij,ij} - {g}_{ii,jj}) = 2 \mu + O(|x|^{-2 - 2q})$, because both sides differ from the scalar curvature by terms quadratic in $\partial_k g_{ij}$ and $\pi_{ij}$.  Since $(\mu - \bar \mu) \in {L^{1}_{-n-q_0/2}}$, it follows that there exists $r_0$ large and depending only on $\epsilon, n, p, q, q_0, \|\mu - \bar \mu\|_{{L^{1}_{-n-q_0/2}}}, \|(g - g_{\mathbb{E}}, \pi)\|_{W^{2, p}_{-q} \times W^{1, p}_{-q-1}}$, and $\|(\bar g-g_{\mathbb{E}}, \bar \pi)\|_{W^{2, p}_{-q} \times W^{1, p}_{-q-1}}$ such that
\[
	 \int_{|x| \ge r} \Big| \sum_{i,j = 1}^n  ({g}_{ij,ij} -{g}_{ii,jj}) - ( \bar{g}_{ij,ij} - \bar{g}_{ii,jj}) \Big| < (n-1) \omega_{n-1 }\epsilon
\] 
 for all $r \geq r_0$.
For the difference of the boundary integrals, note that
\begin{align}
	&\Big|\int_{|x|= r}  \sum_{i,j = 1}^n \left[  (g_{ij,i} - g_{ii,j}) - (\bar{g}_{ij,i} -\bar{g}_{ii,j})\right]  \nu_0^j \Big| \nonumber \\
	&\le  2 \sum_{i, j, k = 1}^n \int_{|x|=r}|\partial_k(g - \bar{g} )_{ij}|  \nonumber \\
	& \le 2 \omega_{n-1} r^{n-1} \sum_{i, j, k = 1}^n  \sup_{|x|=r} |\partial_k(g - \bar g)_{ij}|(x) \nonumber \\
	& \le 2 (\omega_{n-1} r^{n-1}) n^3 (r^{-1 - q}     C_p \|g - \bar g\|_{W^{2, p}_{-q}}) \label{eqn:aux_cont} 
\end{align}
where $C_p$ is the constant that governs the continuous embedding $W^{1, p}_{-1 -q} \subset \C^{1 - \frac{n}{p}}_{-1 - q}$. By choosing $\delta > \|g - \bar g\|_{W^{2, p}_{-q}}$ sufficiently small (depending on $r$) we can ensure that (\ref{eqn:aux_cont}) is less than $(n-1) \omega_{n-1} \epsilon$ so that $|E - \bar E| < \epsilon$. The argument for the linear momentum is similar.
\end{proof}


\subsection{Perturbing to strict dominant energy condition} \label{subse:dec}

We define the constraint map
\begin{equation*}
	\Phi(g,\pi ) =(2\mu, J)=  \left(R_g  - | \pi |_g^2 + \tfrac{1}{n-1} (\tr_g \pi )^2,  \Div_g \pi \right).
\end{equation*} 
We will also use the modified Lie derivative  
\[
\mathcal{L}_g Y= L_Y g - (\Div_g Y) g
\]
of a vector field $Y$.

\begin{lem} \label{le:surjectivity-A}
Let $(g-g_{\mathbb{E}}, \pi) \in W^{2,p}_{-q} \times W^{1,p}_{-1-q}$. The linear map $A: W^{2,p}_{-q} \times W^{1,p}_{-1-q} \rightarrow L^{p}_{-2-q} $ given by 
\[
	A (h,w) = D\Phi |_{(g,\pi)} (h,w) - (0, \tfrac{1}{2} h^{j \ell} J_\ell)
\]
is surjective. 
\end{lem}

The proof is a small modification of the proof of surjectivity of the linearization $D \Phi|_{(g,\pi)}$  in \cite[Proposition 3.1]{Corvino-Schoen:2006}. For the convenience of the reader, we include the argument here.

\begin{proof}
In this proof we treat $\Phi$ as a map defined on the space of $(g, \pi)$ where $g$ is a $(0,2)$-tensor and $\pi$ is a $(2,0)$-tensor, and $(\Div_g \pi)^i := (\pi^{ij})_{;j}$. By direct computation, 
\begin{align*}
	D\Phi|_{(g,\pi)} (h , w) &=\Big( -\Delta_g(\mbox{tr}_g h) + \mbox{div}_g \mbox{div}_g (h) - h^{ij} R_{ij}  -2 h_{ij} \pi_\ell^i \pi^{j\ell}\\
	&\qquad - 2 \pi^j_k w^k_j  +\tfrac{2}{n-1}\mbox{tr}_g \pi (h_{ij} \pi^{ij} + \mbox{tr}_g w), \\
	&\qquad ( \mbox{div}_g w)^i - \frac{1}{2} \pi^{jk} h_{jk;\ell} g^{\ell i} + \pi^{jk} h^i_{j;k} +\tfrac{1}{2} \pi^{ij} (\mbox{tr}_g h)_{,j}\Big).
\end{align*}
Here all indices are raised or lowered with respect to $g$. The formula is well-known and can be found in, for example, \cite[pp.\ 999--1000]{Fischer-Marsden:1973} for $n=3$, but note that the negative divergence operator is used there. Let $(v, Z) \in W^{2,p}_{-q}$ where $v$ is a function and $Z$ is a vector field. Consider $h_{ij} = vg_{ij}$ and $w^{ij} = (\mathcal{L}_g Z)^{ij}$. Then
\begin{align} \label{de:Fredholm} 
	(v, Z) \mapsto A(h, w)
\end{align}
 is a Fredholm operator  from $W^{2,p}_{-q}$ to $L^{p}_{-2-q}$, \emph{cf}. \cite{Bartnik:1986}. It follows that the range of $A$, which contains the range of the operator in \eqref{de:Fredholm}, has finite codimension  in $L^{p}_{-2-q}$. In particular, the range of $A$ is closed. 

   Since $A$ has closed range, we can prove surjectivity of $A$ by showing that the kernel of the adjoint operator $A^*(\xi, Z) = D\Phi|_{(g,\pi)}^*(\xi, Z) - (\frac{1}{2} Z_iJ_j, 0)$ is trivial. One can compute the formal $L^2$-adjoint operator of $D\Phi|_{(g,\pi)}$ 
 \begin{align*}
 	&D\Phi|_{(g,\pi)} ^*(\xi, Z)\\
	 &= \left(  - (\Delta_g \xi) g_{ij} + \xi_{;ij} - \xi R_{ij} +\big( \tfrac{2}{n-1} (\mbox{tr}_g \pi) \pi_{ij} - 2 \pi_{ik} \pi^k_j \big) \xi\right.\\
	& \quad+ \tfrac{1}{2} \left( (L_Z\pi)_{ij} + (Z^k_{;k}) \pi_{ij}- Z_i \pi^k_{j;k} - Z_j\pi^k_{i;k} - Z_{k;m} \pi^{km} g_{ij} - Z_k \pi^{km}_{;m} g_{ij} \right), \\
	&\quad \left. -\tfrac{1}{2} (L_Z g)^{ij} + \big(\tfrac{2}{n-1} (\mbox{tr}_g \pi ) g^{ij}- 2 \pi^{ij}  \big) \xi\right).
 \end{align*}
Let $(\xi, Z)$ be in the dual space $L^{p^*}_{-n+2+q}$ such that $A^*(\xi, Z)  = (0, 0)$.  Taking the trace of the first component of $A^*(\xi, Z)=(0,0)$ gives an equation for $\Delta_g \xi$.  Using this equation, we can eliminate  the term $\Delta_g \xi$ from the system $A^*(\xi, Z)=(0,0)$ to obtain
 \begin{align} \label{eq:system}
\begin{split}
	0&=\xi_{;ij}  - \xi R_{ij}  + \tfrac{1}{n-1} (R_g + 2|\pi|_g^2 - \tfrac{2}{n-1} (\mbox{tr}_g \pi)^2) \xi g_{ij} \\
	&\quad +\big( \tfrac{2}{n-1} (\mbox{tr}_g \pi) \pi_{ij}-2 \pi_{ik} \pi^k_j  \big)\xi\\
	&\quad - \tfrac{1}{2(n-1)} \left(\mbox{tr}_g (L_Z \pi) + Z^k_{;k} \mbox{tr}_g \pi - 2 Z^i \pi^k_{i;k} - Z_{k;m} \pi^{km} - Z_k \pi^{km}_{;m} \right) g_{ij} \\
	 &\quad + \tfrac{1}{2} \left( (L_Z\pi)_{ij} + (Z^k_{;k}) \pi_{ij}- Z_i \pi^k_{j;k} - Z_j\pi^k_{i;k}\right) \\
	 &\quad - \tfrac{1}{2} Z_i J_j + \tfrac{1}{2(n-1)} Z^kJ_k  g_{ij}\\
	0&= \tfrac{1}{2} (L_Z g)^{ij} +\big(  2 \pi^{ij} - \tfrac{2}{n-1} (\mbox{tr}_g \pi ) g^{ij} \big) \xi.
\end{split}
\end{align}
We use a bootstrap argument to show that $(\xi, Z)$ vanishes to infinite order at infinity, i.e. that $|(\xi, Z)| \le C_N |x|^{-N}$ for any integer $N>1$. The initial decay of $\xi$ and $Z$ is on the order of $|x|^{-n+2+q}$. By \eqref{eq:system}, $\nabla^2 \xi$ is of order $|x|^{-n}$ and $L_Z g$ is of order $|x|^{-n+1}$. The following estimates for $n=3$ were proved in \cite[(10)]{Corvino-Schoen:2006} and can be generalized easily to any $n\ge 3$. For any weight $\tau>0$ and for radius $R >0$ large,
\begin{align*}
	& \int_{M\setminus B_R} (\xi |x|^{\tau})^2 |x|^{-n} \le C \int_{M\setminus B_R} (|\nabla^2 \xi| |x|^{2+\tau})^2 |x|^{-n} \\
	& \int_{M\setminus B_R} (|Z| |x|^{\tau})^2 |x|^{-n} \le C \int_{M\setminus B_R} (|L_Zg| |x|^{1+\tau})^2 |x|^{-n}.
\end{align*}
By letting $\tau < n-2$, we conclude that $(\xi, Z)$ is of order $|x|^{-\tau}$ for any $\tau < n-2$, so the decay rate of $(\xi, Z)$ is improved. We can then argue inductively that $(\xi, Z)$ vanishes to infinite order at infinity.

Finally, we take the trace of the first set of equations and  the divergence of the second set of equations of \eqref{eq:system}. The leading order term in the divergence of the second equation is 
\[
	(\mbox{div}_g (L_Zg))^i =  (Z^i_{;k} g^{kj} + Z^j_{;k} g^{ik})_{;j} = \Delta_g Z^i + g^{ik} Z^j_{;kj}.
\]
We replace the term $g^{ik} Z^j_{;kj}$  by terms involving only $\nabla Z, \nabla \xi,$ and  $\xi$, using the trace of the second component $A^*(\xi, Z)$. Hence we obtain a system of linear equations of the form 
\[
	\Delta_g (\xi, Z) = B(x)(\nabla \xi, \nabla Z) + C(x)( \xi, Z),
\] 
where $B, C$ are coefficient matrices. Note that although $B(x)$ and $C(x)$ are different from the coefficient matrices in \cite{Corvino-Schoen:2006}, they do have the same asymptotics. Using a Kelvin transform and a unique continuation argument as in  \cite[pp.\ 196--197]{Corvino-Schoen:2006},  one sees that $(\xi, Z)$ must vanish identically. This completes the proof of surjectivity.
\end{proof}

\begin{notation}
For better readability, we define
\[ s:=\frac{4}{n-2},\]
which will appear frequently in this section.
\end{notation}
For $(u-1, Y) \in W^{2,p}_{-q}$, let
\[
	\tilde{g}= u^s g \quad \mbox{and} \quad  \tilde{\pi} = u^{s/2} (\pi + \mathcal{L}_g Y).
\] 
Define the operator 
\begin{align*} \label{eq:constraint}
	T(u, Y) &= (2 \tilde{\mu} u^{s}, \tilde{J}  u^{s/2}  ), 
\end{align*}
where $(\tilde{\mu}, \tilde{J})$ are the mass and energy densities of $(\tilde{g}, \tilde{\pi})$. Explicitly, the components of $T$ are given by 
\begin{equation} \label{eq:uY}
\begin{split}
	2 \tilde{\mu} u^{s} &=- \tfrac{4(n-1)} {n-2}u^{-1} \Delta_g u + R_{g}+\tfrac{1}{n-1}  \left( \tr_g \pi + \tr_g \mathcal{L}_gY\right)^2   \\
	& \quad - \left( | \pi|_g^2 + 2 (\mathcal{L}_g Y)_{kl} \pi^{kl} + | \mathcal{L}_g Y|_g^2\right)\\
	\tilde{J}_j u^{s/2}&=  (\Div_g \mathcal{L}_g Y+ \Div_g \pi)_{j} + (n-1)\frac{s}{2} u^{-1} u_{,k} (\pi+\mathcal{L}_g Y)^k_j \\
	&\quad - \frac{s}{2} u^{-1} u_{,j} \tr_g(\pi + \mathcal{L}_gY), \qquad j = 1, 2, \ldots, n.
\end{split}
\end{equation}
Here, all indices are raised and lowered with respect to $g$.  

\begin{lem} \label{le:surjectivity}
Suppose that $(g-g_{\mathbb{E}}, \pi) \in W^{2,p}_{-q}\times W^{1,p}_{-1-q}$. For any $(f, V) \in L^{p}_{-2-q}$ there exist $(v, Z) \in W^{2,p}_{-q}$ and symmetric $(0,2)$-tensors $h, w$ in $\C_c^{3,\alpha}$ so that
\[
	DT|_{(1, 0)}(v,Z) + D\Phi|_{(g, \pi)}(h, w) = (f, V_j+\tfrac{1}{2} h^\ell_j J_\ell).
\]

If in addition $(f,V)\in \C^{0,\alpha}_{-n-q_0}$ for some $q_0>0$ and $\alpha \in (0, 1 - \frac{n}{p}]$, and $(g, \pi) \in \C^{2, \alpha}_{loc} \times \C^{1, \alpha}_{loc}$, then we have {$(v, Z)\in \C^{2,\alpha}_{2-n}$}. 

\end{lem}
\begin{proof}
The linearization of $T$ at $(1,0)$ is 
\begin{align} \label{eq:linearization}
\begin{split}
	DT|_{(1, 0)} (v, Z) &= \bigg( -\tfrac{4(n-1)}{n-2} \Delta_g v - 4 Z_{k;\ell} \pi^{k\ell} + \tfrac{2}{n-1} \tr_g \pi \Div_g Z, \\
	& \qquad  \Div_g (\mathcal{L}_g Z)_j + \frac{(n-1)s}{2} v_{,k} \pi^k_j - \frac{s}{2} v_{,j}  \tr_g \pi\bigg) 
\end{split}
\end{align}
where all covariant derivatives and raising of indices are with respect to $g$. For $0<a<(n-2)$ and $p>n$, $DT|_{(1, 0)}: W^{2,p}_{-a} \rightarrow L^{p}_{-2-a}$ is Fredholm with index zero, \emph{cf.\ }\cite{Bartnik:1986}. Because the linear map $A$ defined in Lemma \ref{le:surjectivity-A}  is surjective onto $L^{p}_{-2-q}$ we can find $\C^{3,\alpha}$ compactly supported symmetric tensor fields $\{(h_k, w_k)\}_{k=1}^N$ whose images $A (h_k, w_k)$ span a subspace that complements the image of $DT|_{(1, 0)}$ in $L^{p}_{-2-q}$.
It follows that for every $(f, V) \in L^{p}_{-2-q}$ there exist $(v, Z) \in W^{2,p}_{-q}$ and $(h, w) \in \text{span}\{(h_k, w_k)\}_{k=1}^N$ so that
\begin{equation} \label{eqn:component}
	DT|_{(1, 0)}(v,Z) + D\Phi|_{(g, \pi)}(h, w) = (f, V_j+\tfrac{1}{2} h^\ell_j J_\ell).
\end{equation}
This completes the first part of the proof. 

Now suppose that $(f,V) \in \C^{0,\alpha}_{-n-q_0}$ for some $q_0>0$, and $(g, \pi) \in \C^{2, \alpha}_{loc} \times \C^{1, \alpha}_{loc}$.
We view the system (\ref{eqn:component}) as $(n+1)$ Poisson equations in $v, Z_i$ with a decaying nonhomogeneous term.   Note that the contribution $D \Phi|_{(g, \pi)} (h, w)$ to the nonhomogeneous term is a compactly supported H\"older function.  Using \cite[Theorem 2]{Meyers:1963}, which treats decay of solutions of the Poisson equation, and combining it with weighted Schauder estimates and a bootstrapping argument, we can conclude that $(v,Z)\in \C^{2,\alpha}_{2-n}$.  
 
\end{proof}

\begin{thm} \label{thm:strict-dec}
Let $(M, g, \pi)$ be an asymptotically flat initial data set of type $(p,q,q_0,\alpha)$. 
Assume that the dominant energy condition $\mu \ge | J |_{g}$ holds.  For any $\delta>0$, there exists asymptotically flat initial data $(\bar{g},\bar{\pi})$ of the same type such that 
\[
	\| g- \bar{g}\|_{W^{2,p}_{-q}} \le \delta \qquad \| \pi - \bar{\pi}\|_{W^{1,p}_{-1-q}} \le \delta,
\]
and for some $\gamma>0$ depending on $\delta$, 
\[
	\bar{\mu} > (1+\gamma) |\bar{J}|_{\bar{g}}.
\]
\end{thm}

\begin{proof}
Choose a smooth positive function $f$ such that $f = |x|^{-n-\min(1,q_0)}$ near infinity, and let $(v, Z) \in \C^{2,\alpha}_{2-n}$ and $(h, w) \in \C^{3,\alpha}_c$ be a solution of the system
\[
	DT|_{(1, 0)}(v,Z) + D\Phi|_{(g, \pi)}(h, w) = (2f, \tfrac{1}{2} h^l_j J_l),
\]
whose existence is guaranteed by Lemma \ref{le:surjectivity} (with $V \equiv 0$).  Our goal is to show that for sufficiently small $t>0$, the formula
\begin{equation}\label{bar-family}
	\bar{g} = (1+tv)^s (g + t h) \quad \mbox{and} \quad \bar{\pi} = (1+ t v)^{s/2} (\pi + t \mathcal{L}_g  Z + tw) 
\end{equation}
gives the desired initial data $(\bar{g},\bar{\pi})$ in the statement of the theorem.  Since we clearly have $\| g- \bar{g}\|_{W^{2,p}_{-q}} \le \delta$ and  $\| \pi - \bar{\pi}\|_{W^{1,p}_{-1-q}} \le \delta$ for small enough $t$, it suffices to show that  $\bar{\mu} > (1+\gamma) |\bar{J}|_{\bar{g}}$ for some $\gamma>0$ that depends on $t$.

In the following we denote  $u=1+tv$ and define 
\[
	\Phi_1 (1+tv, tZ, th, tw) = (2 \bar{\mu} u^{s}, \bar{J} u^{s/2}).
\]
By Taylor expansion, 
\begin{align} 
	&\Phi_1 (1+tv, tZ, th, tw) \notag \\
	&= \Phi_1 (1,0,0,0) + tD\Phi_1|_{(1,0,0,0)} (v,Z,h,w) + \mathcal{R} \notag\\
	&= (2\mu, J)+ t DT |_{(1,0)}(v,Z)+ t D\Phi |_{(g,\pi)} (h,w) +  \mathcal{R}\notag \\
	&=(2\mu, J) + t (2f, \tfrac{1}{2} h^k_i J_k)  +  \mathcal{R}, \label{eq:taylor}
\end{align}
where there is a minor abuse of notation in the last line, and the remainder term $\mathcal{R} = \mathcal{R}(x, t)$ has the form
\[
	\mathcal{R} (x, t) = t \int_0^1 \left[ D\Phi_1 |_{(1,0,0,0)+rt(v, Z, h, w)} - D \Phi_1|_{(1, 0, 0, 0)}\right] \, (v, Z, h, w) \, dr.
\]
\begin{claim} We have that  
\[
	|\mathcal{R}(x, t)| \le  C t^2 (1+  |x|)^{2-2n} = O(t^2 f).
\]
where $C$ is a constant that does not depend on $x$ or $t$.
\end{claim} 
\begin{proof}[Proof of claim]
Clearly, $\mathcal{R}(x, t) = O(t^2)$, so it suffices to work outside a large ball $B$ that contains the support of $(h, w)$. Note that when $x \in M \smallsetminus B$ we have that $D \Phi_1|_{(1, 0, 0, 0) + rt (v, Z, h, w)} (v, Z, h, w) = DT_{(1, 0) + rt (v, Z)} (v, Z)$ and hence 
\[
	\mathcal{R} (x, t)= t \int_0^1 \left[ D T |_{(1, 0) +rt (v, Z)} - D T|_{(1, 0)}\right]  \, (v, Z)\, dr.
\]
From (\ref{eq:uY}) we see that the linearization of $T$ at $(u,Y)$ is given by
\begin{align*}  
	&D T |_{(u, Y)} (v, Z) \notag\\
	&= \bigg( \tfrac{4(n-1)}{n-2} (u^{-2} v \Delta_g u -u^{-1} \Delta_g v) \notag \\
	&\qquad + \tfrac{2}{n-1} (\tr_g \pi + \tr_g \mathcal{L}_gY) \tr_g \mathcal{L}_gZ\notag  - 2 (\pi +\mathcal{L}_gY )^{kl} (\mathcal{L}_g Z)_{kl}, \notag \\
	& \qquad \Div_g (\mathcal{L}_g Z)_{j} + (n-1) \frac{s}{2} ( vu^{-1})_{,k} (\pi + \mathcal{L}_g Y)^k_j + (n-1) \frac{s}{2} u^{-1}u_{,k} (\mathcal{L}_g Z)^k_j \notag \\
	&\qquad - \frac{s}{2} (vu^{-1})_{,j} \tr_g (\pi+\mathcal{L}_gY) - \frac{s}{2} u^{-1} u_{,j} \tr_g \mathcal{L}_g Z \bigg).
\end{align*}
The claim follows from substituting $(1, 0) + rt (v, Z)$ and $(1, 0)$ for $(u, Y)$, using that $(v, Z) \in \C^{2, \alpha}_{2 - n}$, and estimating each term in an obvious way.
\end{proof}

From \eqref{eq:taylor} we have 
\[
	u^s \bar{\mu} = \mu + t f + O(t^2 f) \quad \mbox{and} \quad u^{s/2}\bar{J}_i = J_i +t \tfrac{1}{2} h^k_i J_k+ O(t^2 f). 
\]
In particular, for sufficiently small $t>0$, we have
\begin{align} \label{eq:mu}
	u^s \bar{\mu}  >  \mu + \frac{t}{2} f  \quad \mbox{everywhere on } M. 
\end{align}
We claim that $u^{s} |\bar J|_{\bar g} < |J|_g + \frac{tf}{4}$ for $t>0$ sufficiently small.  Choose $f_1$ to be a smooth nonnegative function whose support is larger than that of $h$. Then 
\[ \bar{g}^{ij} = u^{-s}(g^{ij} -th^{ij}+O(t^2 f_1)).\]
Therefore
\begin{align*} 
(u^{s} |\bar J|_{\bar g})^2 
&= u^{2s} \bar{g}^{ij} \bar{J}_i\bar{J}_j\\
&= (g^{ij} -th^{ij}+O(t^2 f_1))( J_i +t \tfrac{1}{2} h^k_i J_k+ O(t^2 f)) ( J_j +t \tfrac{1}{2} h^k_j J_k+ O(t^2 f))\\
&= g^{ij}J_i J_j + t (-h^{ij} J_i J_j + \tfrac{1}{2}  g^{ij} h^k_i J_k J_j + \tfrac{1}{2} g^{ij}J_i  h^k_j J_k)\\
&\quad
+ O(t^2 |J|^2 f_1 +  t^2 |J|f+ t^3 f_1+ t^4 f^2) \\ 
&= |J|_g^2
+ O( t^2 |J|f+ t^3 f^2) \\ 
&= \left(|J|_g + \frac{tf}{4}\right)^2 - \frac{tf}{2}|J|_g - \frac{t^2 f^2}{16}  + O( t^2 |J|f+ t^3 f^2) \\ 
&<\left(|J|_g + \frac{tf}{4}\right)^2,
\end{align*}
where we choose $t>0$ to be sufficiently small in the last line, proving the claim.  Observe that the computation above explains the motivation behind the definition of $A$.  The $\frac{1}{2}h^l_j J_l$ term in $A$ is chosen specifically so that the first order change in $|J|_g$ under the deformation vanishes.

Now fix $t>0$ small enough so that $u^s \bar \mu > \mu + \frac{tf }{2} $ and $u^s |\bar J|_{\bar g}  < |J|_g + \frac{tf}{4}$. Our assumptions imply that  $\sup_{ M} \frac{|J|_g}{f} < \infty$. 
For $x \in M$ such that $|\bar J|_{\bar g} (x) \neq 0$ it follows that 
\begin{eqnarray*}
\frac{\bar \mu}{ |\bar J|_{\bar g}} = \frac {u^s \bar \mu}{u^s |\bar J|_{\bar g}} > \frac{\mu + t f /2}{|J|_g + t f /4} \geq \frac{|J|_g + tf/2}{|J|_g + tf/4} \geq 1 + \gamma
\end{eqnarray*} where $\gamma^{-1} := 1 + \frac{4}{t} \sup_{M} \frac{|J|_g}{f}$. Since $\bar \mu >0$ we conclude that $\bar \mu > (1 + \gamma) |\bar J|_{\bar g}$ on $M$, as desired.  Finally, note that $\gamma$ depends only on $t$ and $\sup_M \frac{|J|_g }{ f}$.
\end{proof}

\begin{remark} The assertion of Lemma 1 in \cite{Schoen-Yau:1981-pmt2} is similar to but weaker than that of Theorem \ref{thm:strict-dec}. We also note that the proof of this lemma contains an error. 
\end{remark}


\subsection{Harmonic asymptotics} \label{subse:harmonic-asymptotics}

We first show that any asymptotically flat initial data set can be slightly perturbed so that
\begin{align*}
	g = u^{s} g_{\mathbb{E}}, \quad \pi = u^{s/2} \mathcal{L}_{g_{\mathbb{E}}} Y
\end{align*} 
outside a compact set, for some choice of $u$ and $Y$, meanwhile prescribing any constraints $(\mu, J)$ that are close to the original ones. We then show that $(g,\pi)$ has harmonic asymptotics if the imposed constraints decay fast enough.

The first result is a modification of \cite[Theorem 1]{Corvino-Schoen:2006}, where the special case of three-dimensional vacuum initial data was treated. 

\begin{lem} \label{le:Corvino-Schoen:2006}
Suppose $(g -g_{\mathbb{E}}, \pi)\in W^{2,p}_{-q}\times W^{1,p}_{-1-q} $. There exist $C_0, \delta_0 > 0$ so that given $(\bar{\mu},\bar{J}) \in L^{p}_{- 2 - q}$ with $\|(\bar{\mu}-\mu,\bar{J}- J)\|_{L^{p}_{-2-q}} \le \delta \le \delta_0$, there is an  initial data set $\left(\bar{g}, \bar{\pi} \right)$ such that
\begin{align} 
	\bar{g} = u^s g_{\mathbb{E}}, \quad \bar{\pi} = u^{s/2} \mathcal{L}_{g_{\mathbb{E}}} Y
\end{align}
outside a compact set for some $(u-1, Y)\in W^{2,p}_{-q}$, and such that the mass and current densities of $(\bar{g},\bar{\pi})$ are $(\bar{\mu}, \bar{J})$.  Moreover, 
\[
	\| \bar{g} - g\|_{W^{2,p}_{-q}} \leq C_0 \delta, \qquad \| \bar{\pi} -\pi \|_{W^{1,p}_{-1-q}} \leq C_0 \delta.
\] 
\end{lem}
\begin{proof}
For $\lambda \geq 1$ large define the cut-off initial data 
\[
	\hat{g}_\lambda = \chi_\lambda g + (1 - \chi_\lambda) g_{\mathbb{E}}, \qquad \hat{\pi}_\lambda = \chi_\lambda \pi
\]
where $\chi_\lambda (x) = \chi(x/\lambda) $ and $\chi$ is a smooth cut-off function on $\mathbb{R}^n$ that is $1$ on $\{ |x| \le 1\}$ and $0$ on $\{ |x| \ge 2\}$. Note that $\| (\hat g_\lambda - g, \hat \pi_\lambda - \pi)\|_{W^{2,p}_{-q}\times W^{1,p}_{-1-q} } \to 0$ as $\lambda \to \infty$. 

In the following, we suppress the subscript $\lambda$ when the context is clear. Define
\begin{align*}
	\tilde{g} = u^{s} \hat{g},  \qquad \tilde{\pi} = u^{s/2} (\hat{\pi} + \mathcal{L}_{\hat g} Y),
\end{align*}
where $(u - 1, Y) \in W^{2,p}_{-q}$.  Let $(\tilde{\mu}, \tilde{J})$ be the mass and current densities of $(\tilde{g}, \tilde{\pi})$. Define the map $T_{(\hat{g}, \hat{\pi})}(u,Y)= (2 \tilde{\mu}, \tilde{J} )$ from $(W^{2,p}_{-q}+1)\times W^{2,p}_{-q} \rightarrow L^{p}_{-2-q}$. The linearization of $T_{(\hat{g}, \hat{\pi})}$ at $(1,0)$ is 
\begin{align*}
	&DT_{(\hat{g}, \hat{\pi})}|_{(1, 0)}(v, Z) \\
	&= \bigg(-\tfrac{4(n-1)} {n-2}\Delta_{\hat{g}} v - s \big[R_{\hat{g}} - | \hat{\pi}|_{\hat{g}}^2 + \tfrac{1}{n-1} (\tr_{\hat{g}} \hat{\pi})^2\big] v \\
	&\qquad - 4 Z_{k;l} \hat{\pi}^{kl} + \tfrac{2}{n-1} \tr_{\hat{g}} \hat{\pi} \,\Div_{\hat{g}} Z, \\
	&\qquad \Div_{\hat{g}} (\mathcal{L}_{\hat{g}}Z)_j+ (n-1)\frac{s}{2} v_{,k} \hat{\pi}^k_j - \frac{s}{2} v_{,j} \tr_{\hat{g}} \hat{\pi} - \frac{s}{2} (\Div_{\hat{g}} \hat{\pi} )_j v \bigg),
\end{align*}
where indices are raised and covariant derivatives are taken with respect to $\hat{g}$. The map $T_{(g, \pi)} : (W^{2,p}_{-q}+1)\times W^{2,p}_{-q} \rightarrow L^{p}_{-2-q}$ is defined analogously. Because $q \in ((n-2)/2, n-2)$ and $p>n$, $DT_{(\hat{g}, \hat{\pi})}|_{(1, 0)}$ and $D T_{(g, \pi)}|_{(1, 0)}$ are Fredholm operators of index $0$ for $\lambda$ sufficiently large (see \cite{Bartnik:1986}). 

Let $K_1$ be a complementing subspace for the kernel of $D T_{(g, \pi)}|_{(1, 0)}$ in $W^{2, p}_{-q} \times W^{1, p}_{-2-q}$. Since the linearization $D\Phi|_{(g, \pi)}: W^{2,p}_{-q}\times W^{1,p}_{-1-q} \to L^{p}_{-2-q}$ is surjective (see  \cite{Corvino-Schoen:2006} and  Lemma \ref{le:surjectivity-A}) and because $DT_{(g, \pi)}|_{(1, 0)}$ is Fredholm, we can find $\C^{3,\alpha}$ compactly supported symmetric $(0, 2)$-tensors $\{(h_k, w_k)\}_{k=1}^N$ whose images $D\Phi|_{(g, \pi)}(h_k, w_k)$ form a basis for a complementing subspace of the image of $DT_{(g, \pi)}|_{(1, 0)}$ in $L^{p}_{-2-q}$. Let $K_2 = \mbox{span} \{ (h_k, w_k)\}_{k=1}^N$. We define the maps $\overline{T}_{(\hat{g}, \hat{\pi})}, \overline{T}_{(g, \pi)}: K_1  \times K_2 \rightarrow L^{p}_{-2-q}$ by 
\[
	\overline{T}_{(\hat{g}, \hat{\pi})}(u,Y, h,w)  = \Phi( u^s \hat{g}+h, u^{s/2} (\hat{\pi}+\mathcal{L}_{\hat{g}} Y )+w)
\] 
and 
\[
	\overline{T}_{({g},{\pi})}(u,Y, h,w)  = \Phi( u^s {g}+h, u^{s/2} ({\pi}+\mathcal{L}_{g} Y )+w). 
\] 

The maps $\overline{T}_{(\hat g, \hat \pi)}$ and $\overline{T}_{(g, \pi)}$ are continuously differentiable. Using that $(\hat g, \hat \pi)$ converges to $(g, \pi)$ in $W^{2, p}_{-q} \times W^{1, p}_{-q-1}$ as $\lambda \to \infty$ it is easy to see that $D \overline{T}_{(\hat g, \hat \pi)}|_{(u, Y, h, w)}$ converges to  $D \overline{T}_{(g, \pi)}|_{(u, Y, h, w)}$ as $\lambda \to \infty$ locally uniformly in $(u, Y, h, w) \in K_1 \times K_2$ in the strong operator topology. Observe that $D \overline T_{(g, \pi)}|_{(1, 0, 0, 0)}$ is an isomorphism by construction. We conclude from the inverse function theorem that there exists $\delta_0>0$ such that for all $\lambda \geq 1$ sufficiently large, $\overline{T}_{(g, \pi)}$ restricts to a $\C^1$ diffeomorphism defined on an open neighborhood of $(1, 0, 0, 0)$ (independent of $\lambda \geq 1$) in $K_1 \times K_2$ and onto an open neighborhood containing the $L^{p}_{-2-q}$ ball of radius $2 \delta_0$ centered at $(2 \hat \mu, \hat J)$. 

Using that $\overline{T}_{(g, \pi)}(1,0, 0,0)=(2\mu,J)$ and that $\|(\mu, J) - (\hat \mu, \hat J)\|_{L^{p}_{-2-q}} \to 0$ as $\lambda \to \infty$ we see that $\|(\bar \mu, \bar J) - (\hat \mu, \hat J)\|_{{L^{p}_{-2-q}}} < 2 \delta$ provided that $\|(\bar \mu, \bar J) - (\mu, J)\|_{L^{p}_{-2-q}} < \delta$ and $\lambda \geq 1$ is sufficiently large. Hence if $\delta \in (0, \delta_0)$ there exist $(u-1, Y) \in W^{2, p}_{-q}$ and $\C^{3,\alpha}$ compactly supported symmetric $(0, 2)$-tensors $(h, w)$ such that $\overline{T}_{(\hat g, \hat \pi)} (u, Y, h, w) = (2\bar \mu, \bar J)$ and such that
\[
	\|(u-1,Y)\|_{W^{2,p}_{-q}}\le C_1 \delta, \qquad \| (h,w)\|_{W^{2,p}_{-q} \times W^{1,p}_{-1-q}} \le C_1 \delta
\]
for some constant $C_1>0$ that only depends on $(g,\pi)$. 
\end{proof}

The following proposition is straightforward except for a subtlety that arises when $n=3$.

\begin{prop} \label{pr:harmonic}
Suppose that $(M, g, \pi)$ is an asymptotically flat initial data set of type $(p,q,q_0,\alpha)$  with $q_0 > 1$ (rather than just $q_0 > 0$) and such that
\begin{align} \label{eq:harmonic}
	g = u^{s} g_{\mathbb{E}}, \quad \pi = u^{s/2} \mathcal{L}_{g_{\mathbb{E}}} Y ,
\end{align} 
outside a large ball $B$ for some $(u-1, Y) \in W^{2,p}_{-q}$.  Then $(g,\pi)$ has harmonic asymptotics in the sense of Section \ref{section-definitions}.
\end{prop}

\begin{proof}

By \eqref{eq:uY}, outside a compact set, $(u-1,Y)$ satisfies
\begin{align*}
	 \tfrac{4(n-1)} {n-2}u^{-1} \Delta_{g_{\mathbb{E}}} u - \tfrac{1}{n-1}  \left( \tr_{g_{\mathbb{E}}} \mathcal{L}_{g_{\mathbb{E}}}Y\right)^2 +  | \mathcal{L}_{g_{\mathbb{E}}} Y|^2 &= -2 \mu u^{s}\\
	\Delta_{g_{\mathbb{E}}} Y_j+ (n-1)\frac{s}{2} u^{-1} u_{,k} (\mathcal{L}_{g_{\mathbb{E}}} Y)^k_j - \frac{s}{2} u^{-1} u_{,j} \tr_{g_{\mathbb{E}}}( \mathcal{L}_{g_{\mathbb{E}}}Y) &= J_j u^{s/2},
\end{align*}
for $j = 1,2, \ldots, n$.  By Sobolev embedding, $(u-1,Y) = O^{1, \alpha}(|x|^{-q})$.  
As in the proof of Lemma \ref{le:surjectivity}, we may view this system as $n+1$ Poisson equations for the functions $u, Y_i$, so that $\Delta_{g_{\mathbb{E}}} (u, Y) = O^{0, \alpha}(|x|^{\max(-2q-2, -n-q_0)})$.  Then using the fact that $q \in ((n-2)/2, n-2)$, Theorem 2 of \cite{Meyers:1963} and weighted Schauder estimates imply that 
\begin{align}
\begin{split}
	u &= 1 + a|x|^{2-n} + O^{2, \alpha}(|x|^{2-n-\gamma})\\
	Y_i& = b_i |x|^{2-n} +O^{2, \alpha}(|x|^{2-n-\gamma})
\end{split}
\end{align}
for constants $a, b_i$, for any $0<\gamma<\min(2q+2-n, q_0) $.  It follows that $(u-1, Y)= O^{2,\alpha}(|x|^{2-n})$. 

For $n>3$, since $\max(2-2n,-n-q_0)< -n-1$, we may repeat the above argument with improved decay of the source terms to conclude the desired result. 

When $n=3$, we argue as follows. We expand the nonlinear source terms above.  We find that
\begin{align*}
 |\mathcal{L}_{g_{\mathbb{E}}} Y |^2 &= |L_Y g|^2 -2\tr_{g_{\mathbb{E}}} (L_g Y)(\Div_{g_{\mathbb{E}}}Y) + 3(\Div_{g_{\mathbb{E}}} Y)^2\\
 & = 2\left[\sum_{i,j=1}^3 Y_{i,j}^2 + Y_{i,j}Y_{j,i}\right] - (\Div_{g_{\mathbb{E}}} Y)^2.
 \end{align*}
Since
\[ Y_{i,j} = - b_i \frac{x_j}{|x|^{3}} +O^{1, \alpha}(|x|^{-2-\gamma}),
\]
we have 
\[ (\tr_{g_{\mathbb{E}}} \mathcal{L}_{g_{\mathbb{E}}} Y)^2=(\Div_{g_{\mathbb{E}}} Y)^2 = \left(\sum_{i=1}^3 Y_{i,i}\right)^2 = \frac{(b\cdot x)^2}{  |x|^{6}} + O^{1, \alpha}(|x|^{-4-\gamma})\]
and  
\[\sum_{i,j=1}^3 Y_{i,j}^2 + Y_{i,j}Y_{j,i}= \frac{ |b|^2}{ |x|^{4}} + \frac{(b\cdot x)^2}{|x|^6} 
+ O^{1, \alpha}(|x|^{-4-\gamma}).\]
Thus the constraint equation for $\mu$ becomes
\[8\Delta_{g_{\mathbb{E}}} u =   - \frac{ 2|b|^2}{ |x|^{4}}-  \frac{(b\cdot x)^2}{2|x|^6} 
+ O^{0, \alpha}(|x|^{-4-\gamma}),\]
as long as $\gamma \leq q_0 -1$.  Now observe that
 \begin{align*}
 	\Delta_{{g_{\mathbb{E}}}}  \frac{ x_i x_j}{|x|^4} &= \frac{2 \delta_{ij}}{|x|^4} - \frac{4 x_i x_j}{|x|^6},\\
	 \Delta_{g_{\mathbb{E}}}  \frac{1}{ |x|^2}&= \frac{2 }{|x|^4}.
 \end{align*}
For appropriate choice of constants $C_0$ and $C_{ij}$, we have
 \[
	\Delta_{g_{\mathbb{E}}}  \left( u - C_0\frac{1}{|x|^2} - C_{ij}\frac{ x_i x_j}{|x|^4}\right) =  O^{0, \alpha}(|x|^{-4-\gamma}).
 \]
Now apply \cite[Theorem 2]{Meyers:1963} and weighted Schauder estimates to obtain the desired result for $u$. The argument for $Y$ is analogous.  
\end{proof}

\subsection {Proof of Theorem \ref{th:density-theorem-section}}

Using Theorem \ref{thm:strict-dec} and Proposition \ref{prop:cont} we may reduce to the case where $\mu > (1 + \gamma )|J|_g$. For $l \geq 1$ define $\xi_\lambda(x) := \xi(\frac{x}{\lambda})$ where 
\begin{align*}
	\xi(x)  = \left\{ \begin{array}{ll} 
	1 \quad &\mbox{when  } |x| \le 1\\ 
	 |x|^{-1} & \mbox{when } |x| \ge 2  \end{array}\right..
\end{align*}
Note that  $\xi_\lambda (\mu, J)\in \C^{0,\alpha}_{-n-q_0-1}$ and that $\| \xi_\lambda (\mu, J) - (\mu, J)\|_{L^{p}_{-2-q}}$ tends to zero as $\lambda \to \infty$. By Lemma \ref{le:Corvino-Schoen:2006} and Proposition \ref{pr:harmonic} there are initial data $(g_\lambda, \pi_\lambda)$ with harmonic asymptotics and mass and current densities $ \xi_\lambda (\mu, J)$ 
such that $\| (g - g_\lambda, \pi - \pi_\lambda) \|_{W^{2,p}_{-q}\times W^{1,p}_{-1-q}} \to 0$ as $\lambda \to \infty$. In particular, $\|g - g_\lambda \|_{\C^0} \to 0$ so that 
\[
	|\xi_\lambda J|^2_{g_\lambda} = g_\lambda^{ij} \xi_\lambda^2 J_i J_j = \xi^2_\lambda |J|^2_g (1 + o(1)) \quad \mbox{as}\quad  \lambda \to \infty.
\]
It follows that
\[
	\xi_\lambda \mu > \xi_\lambda (1+\gamma)|J|_g \geq \left(1 + \frac{\gamma}{2}\right) |\xi_\lambda J|_{g_\lambda}
\] 
for $\lambda$ sufficiently large. In particular, the strict dominant energy condition holds for $(g_\lambda, \pi_\lambda)$. The convergence of mass and linear momentum as $\lambda \to \infty$ follows from Proposition  \ref{prop:cont}.  \qed 

\begin{remark}  It is a subtlety in the proof of Theorem \ref{th:density-theorem-section} above that the decay conditions for $(\mu, J)$ improve when we are reducing to harmonic asymptotics. In fact, if we are content to preserve the dominant energy condition without requiring strictness, we can arrange for $(\bar{\mu}, \bar{J})$ to vanish outside a large compact set by choosing $\xi$ with compact support. The modified initial data will then be smooth at infinity. 
\end{remark} 

\begin {remark}
The assertion in \cite{Eichmair:2008} that $\tr_g (k) = O (|x|^{-n})$ for $n$-dimensional initial data sets with harmonic asymptotics is wrong, as was pointed out to the first named author by A. Carlotto. When $n > 3$, the arguments in \cite{Eichmair:2008} continue to hold. When $n =3$, the additional assumption that $\tr_g (k) = O (|x|^{- \gamma})$ for some $\gamma > 2$ in Theorem 3 of \cite{Eichmair:2008} is required throughout (rather than just for the rigidity case). Of course, this is exactly the case treated by R. Schoen and S.-T. Yau in \cite{Schoen-Yau:1981-pmt2}.
\end {remark}

\begin{remark}
When $3 < n < 8$, the full positive mass theorem $E \geq |P|$ can be obtained from the positive energy theorem $E \geq 0$ in the form \cite[Theorem 3] {Eichmair:2008} using the following reduction argument. Assume that $0 < E < |P|$. Fix $\theta \in (0, 1)$ such that $E' := \frac{E - \theta |P|}{1 - \theta^2} < 0$. Using the preceding remark we may assume that $(M, g, k)$ is smooth (including decay on derivatives) outside a large compact subset of $M$ and that $\mu = 0$ and $J = 0$ there. According to \cite[Theorem 6.1]{CO:1981}, there exists a (vacuum) spacetime development of the asymptotically flat end of $(M, g, k)$ in which we may deform the end of $(M, g, k)$ to a boosted slice (of slope $\theta$) that has energy $E'$. See \cite[p. L115]{Chrusciel:1986} for a remark on the transformation behavior of the energy-momentum tensor. The deformed initial data $(M', g', k')$ satisfies the conditions of the positive energy theorem in \cite[Theorem 3]{Eichmair:2008}. A contradiction. 

We are grateful to P.T. Chrus\'ciel for useful discussions related to this argument. This reduction is folklore in the mathematical relativity community when the initial data set is given as a spacelike slice of an asymptotically flat spacetime, see e.g. \cite{Chrusciel:1989}.
\end{remark}

{\bf Acknowledgements. } The first named author acknowledges the support of NSF grant DMS-0906038 and of SNF grant 200021-140467.  The second named author acknowledges the support of NSF grant DMS-1005560 and DMS-1308837.  The third named author acknowledges the support of NSF grant DMS-0903467.  The last named author acknowledges the support of NSF grants DMS-1105323 and DMS-1404966.

\bibliographystyle{amsplain}
\bibliography{spacetimePMTreferences}
\end{document}